\documentclass{amsart}
\usepackage{amssymb,
enumitem,
mathrsfs,
hyperref,
tikz,
upgreek,
subcaption,
graphics
}
\usetikzlibrary{decorations.pathreplacing}

\hypersetup{colorlinks=true}
\numberwithin{equation}{section}
\usepackage[T1]{fontenc}
\usepackage[utf8]{inputenc}
\usepackage{todonotes}

\usepackage{xargs}
\usepackage{xspace}

\usepackage[capitalise]{cleveref}
\usepackage{scalerel,stackengine}
\usepackage[all,cmtip]{xy} 

\DeclareFontFamily{U}{mathb}{\hyphenchar\font45}
\DeclareFontShape{U}{mathb}{m}{n}{
      <5> <6> <7> <8> <9> <10> gen * mathb
      <10.95> matha10 <12> <14.4> <17.28> <20.74> <24.88> mathb12
      }{}
\DeclareSymbolFont{mathb}{U}{mathb}{m}{n}

\DeclareMathSymbol{\sqsubsetneq}{3}{mathb}{"88}

\newcommand{\ball}[2]{\shade[ball color=black!30!white] (#1,#2,0) circle (.3cm)}

\newcommand{\RR}{\mathbb{R}}

\newcommand{\pre}[2]{#2^{#1}}
\newcommand{\On}{\mathrm{Ord}}

\newcommand{\Nbhd}{\boldsymbol{N}}

\newcommand{\pow}{\mathscr{P}}

\newcommand{\lh}{\operatorname{lh}}

\newcommand{\ZF}{{\sf ZF}}

\newcommand{\AC}{{\sf AC}}
\newcommand{\AD}{{\sf AD}}
\newcommand{\DC}{{\sf DC}}

\newcommandx{\set}[2][2 = undefined]{
  \ifthenelse{\equal{#2}{undefined}}{\{ #1 \}}{
    \{ #1 \mid #2 \}
  }
}

\DeclareMathOperator{\cbtype}{tp}

\newcommand{\countablecof}{\mathsf{Cof}_\omega}

\newcommand{\baire}{\mathcal{N}}
\newcommand{\cantor}{\mathcal{C}}
\newcommand{\nsdclass}{\boldsymbol{\Gamma}}
\newcommand{\NSD}{\operatorname{NSD}}

\newcommand{\nsdc}[1]{\mathrm{Class}_{\NSD}(#1)}
\newcommand{\nsds}[1]{\mathrm{Set}_{\NSD}(#1)}
\newcommand{\SD}{\operatorname{SD}}

\newcommand{\sds}[1]{\mathrm{Set}_{\SD}(#1)}

\newcommand{\sdc}[1]{\mathrm{Class}_{\SD}(#1)}
\newcommand{\wc}[1]{\mathrm{Class}_{\w}(#1)}

\newcommand{\kerp}[1]{\ker_{\cb}(#1)}

\newcommand{\cb}{\operatorname{CB}}
\newcommand{\rkcb}[1]{\rank{#1}_{\cb}}
\newcommand{\rkcbx}[2]{\rank{#1,#2}_{\cb}}

\newcommand{\derrkcb}[2]{D^{#1}_{\cb}(#2)}

\newcommand{\ptgl}{\operatorname{ptgl}}

\newcommand{\ptglsq}[1]{\ptgl(#1)}

\newcommand{\comp}{\operatorname{Comp}}

\newcommand{\wadge}[1]{\mathscr{W}_{#1}}

\newcommand{\conc}{{^{\smallfrown}}}

\newcommand{\w}{\mathsf{W}}
\newcommand{\lew}{\le_{\w}}
\newcommand{\slew}{<_{\w}}
\newcommand{\eqw}{\equiv_{\w}}
\newcommand{\wadgeclass}[1]{#1 {\downarrow}}
\newcommand{\wadgedegree}[1]{[#1]_\w}
\newcommand{\rank}[1]{||#1||}
\newcommand{\wrank}[1]{||#1||_\w}

\newcommand{\Diff}{D}
\newcommand{\differenceseq}[2]{\Diff_{#1}(#2)}
\newcommand{\difference}[2]{\Diff_{#1}(#2)}
\newcommand{\differenceopen}[1]{\Diff_{#1}(\boldsymbol{\Sigma}^0_1)}
\newcommand{\differenceopenX}[2]{\Diff_{#1}(\boldsymbol{\Sigma}^0_1(#2))}
\newcommand{\differencestaropenX}[2]{\Diff_{#1}^*(\boldsymbol{\Sigma}^0_1(#2))}
\newcommand{\differencesigmaalphaX}[3]{\Diff_{#1}(\boldsymbol{\Sigma}^0_{#2}(#3))}
\newcommand{\differenceopenXcheck}[2]{\widecheck{\Diff}_{#1}(\boldsymbol{\Sigma}^0_1(#2))}

\newcommand{\cof}[1]{\operatorname{cof}(#1)}

\newcommand{\clopensets}[1]{\boldsymbol{\Delta}^0_1(#1)}

\newcommand{\dercomp}[2]{D^{#1}_{\comp}(#2)}
\newcommand{\rkcomp}[1]{\rank{#1}_{\comp}}

\renewcommand{\ker}{\operatorname{ker}}

\newcommand{\inftycostbaire}[1]{(#1)^{\infty}}

\stackMath
\newcommand\widecheck[1]{%
\savestack{\tmpbox}{\stretchto{%
  \scaleto{%
    \scalerel*[\widthof{\ensuremath{#1}}-\widthof{\ensuremath{#1}}/4]{\kern-.3pt\rotatebox[origin=c]{180}{$\bigwedge$}\kern-.3pt}%
    {\rule[-\textheight/2]{1ex}{\textheight}}
  }{\textheight}%
}{0.5ex}}%
\stackon[1pt]{#1}{\tmpbox}%
}
\newcommand{\sdordinal}[1]{\On_{\mathrm{SD}}(#1)}
\newcommand{\nsdordinal}[1]{\On_{\mathrm{NSD}}(#1)}

\newenvironment{enumerate-(a)}{\begin{enumerate}[label={\upshape (\alph*)}, leftmargin=2pc]}{\end{enumerate}}
\newenvironment{enumerate-(a)-r}{\begin{enumerate}[label={\upshape (\alph*)}, leftmargin=2pc,resume]}{\end{enumerate}}
\newenvironment{enumerate-(a)-5}{\begin{enumerate}[label={\upshape (\alph*)}, leftmargin=2pc,start=5]}{\end{enumerate}}
\newenvironment{enumerate-(A)}{\begin{enumerate}[label={\upshape (\Alph*)}, leftmargin=2pc]}{\end{enumerate}}
\newenvironment{enumerate-(A)-r}{\begin{enumerate}[label={\upshape (\Alph*)}, leftmargin=2pc,resume]}{\end{enumerate}}
\newenvironment{enumerate-(i)}{\begin{enumerate}[label={\upshape (\roman*)}, leftmargin=2pc]}{\end{enumerate}}
\newenvironment{enumerate-(i)-r}{\begin{enumerate}[label={\upshape (\roman*)}, leftmargin=2pc,resume]}{\end{enumerate}}
\newenvironment{enumerate-(I)}{\begin{enumerate}[label={\upshape (\Roman*)}, leftmargin=2pc]}{\end{enumerate}}
\newenvironment{enumerate-(I)-r}{\begin{enumerate}[label={\upshape (\Roman*)}, leftmargin=2pc,resume]}{\end{enumerate}}
\newenvironment{enumerate-(1)}{\begin{enumerate}[label={\upshape (\arabic*)}, leftmargin=2pc]}{\end{enumerate}}
\newenvironment{enumerate-(1)-r}{\begin{enumerate}[label={\upshape (\arabic*)}, leftmargin=2pc,resume]}{\end{enumerate}}
\newenvironment{itemizenew}{\begin{itemize}[leftmargin=2pc]}{\end{itemize}}
\newenvironment{enumerate-a-(1)}{\begin{enumerate}[label={\upshape a.\arabic*)}, leftmargin=2pc]}{\end{enumerate}}
\newenvironment{enumerate-b-(1)}{\begin{enumerate}[label={\upshape b.\arabic*)}, leftmargin=2pc]}{\end{enumerate}}

\newtheorem*{theorem*}{Theorem}
\newtheorem*{lemma*}{Lemma}
\newtheorem*{proposition*}{Proposition}
\newtheorem{theorem}{Theorem}[section]
\newtheorem{lemma}[theorem]{Lemma}
\newtheorem{corollary}[theorem]{Corollary}
\newtheorem{proposition}[theorem]{Proposition}
\newtheorem{question}[theorem]{Question}

\newtheorem{maintheorem}{Main Theorem}

\theoremstyle{definition}
\newtheorem{defin}[theorem]{Definition}
\newtheorem{example}[theorem]{Example}

\newtheorem{fact}[theorem]{Fact}
\theoremstyle{remark}
\newtheorem{remark}[theorem]{Remark}

\begin{document}

\title[Wadge hierarchy on Polish zero-dimensional spaces]{A classification of the Wadge hierarchies\\
on zero-dimensional Polish spaces}
\date{\today}
\author[R.~Carroy]{Raphaël Carroy}
\address{Dipartimento di matematica \guillemotleft{Giuseppe Peano}\guillemotright, Universit\`a di Torino, Via Carlo Alberto 10, 10121 Torino --- Italy}
\email{raphael.carroy@unito.it}

\author[L.~Motto Ros]{Luca Motto Ros}
\address{Dipartimento di matematica \guillemotleft{Giuseppe Peano}\guillemotright, Universit\`a di Torino, Via Carlo Alberto 10, 10121 Torino --- Italy}
\email{luca.mottoros@unito.it}

\author[S.~Scamperti]{Salvatore Scamperti}
\address{Dipartimento di matematica \guillemotleft{Giuseppe Peano}\guillemotright, Universit\`a di Torino, Via Carlo Alberto 10, 10121 Torino --- Italy}
\email{salvatore.scamperti@unito.it}
\thanks{This work was supported by the project PRIN 2017 ``Mathematical Logic: models, sets, computability'' (prot.\ 2017NWTM8R) and by  the ``National Group for Algebraic and Geometric Structures, and their Applications'' (GNSAGA -- INDAM)}
\thanks{The authors are grateful to the anonymous referee for their useful suggestions and references}

\begin{abstract}
We provide a complete classification, up to order-isomorphism, of all possible Wadge hierarchies on zero-dimensional Polish spaces using (essentially) countable ordinals as complete invariants. We also observe that although our assignment of invariants is very simple  and there are only \( \aleph_1 \)-many equivalence classes, the above classification problem is quite complex from the descriptive set-theoretic point of view: in particular, there is no Borel procedure to determine whether two zero-dimensional Polish spaces have isomorphic Wadge hierarchies. All results are based on a complete and explicit description of  the Wadge hierarchy on an arbitrary zero-dimensional Polish space, depending on its topological properties.
\end{abstract}

\maketitle

\section{Introduction}

Work in \( \ZF + \DC(\RR) \).%
\footnote{\( \DC(\RR) \) is the axiom of Dependent Choices over \( \RR \), which is equivalent to having \( \DC(X) \) for all Polish spaces \( X \).}
Given topological spaces $X$ and $Y$, a \emph{continuous reduction} from $A\subseteq X$ to $B\subseteq Y$ is a continuous function $f \colon X\to Y$ satisfying $f^{-1}(B) = A$. When such an \( f \) exists 
we say that $A\subseteq X$ \emph{continuously reduces} or \emph{Wadge reduces} to $B\subseteq Y$, and we write $A\lew^{X,Y}B$, or $A\lew^{X}B$ if $X=Y$.
Notice that \( A \lew^{X,Y} B \) if and only if \( {X \setminus A} \lew^{X,Y} {X \setminus B} \).
We also denote by \( \slew^{X ,Y}\) the strict part of \( \lew^{X,Y} \), that is: \( A \slew^{X,Y} B \iff { {A \lew^{X,Y} B} \wedge {B \not\lew^{Y,X} A}}\). Similarly, we write \( A \eqw^{X,Y} B \iff {{A \lew^{X,Y} B} \wedge {B \lew^{Y,X} A}} \).
To simplify the notation, we again write \( \slew^X \) and \( \eqw^X \) instead of \( \slew^{X,X} \) and \( \eqw^{X,X} \), respectively.

Continuous reducibility is a transitive and reflexive relation, that is a \emph{quasi-order}, and \( \eqw^X \) is the equivalence relation canonically induced by \( \lew^X \). We call \emph{Wadge degree} of $A\subseteq X$ the $\eqw^{X}$-equivalence class of $A$, and we denote it by $\wadgedegree{A}^X$.
The Wadge reducibility \( \lew^X \) induces a partial order on Wadge degrees of $X$ that we call \emph{Wadge hierarchy} on $X$ and denote by $\wadge{X}$. When \( \lew^X \) is well-founded, we can associate to 
each set \( A \subseteq X \) (or to its Wadge degree 
\( \wadgedegree{A}^{X} \)) its rank according to 
\( \lew^{X} \), which is called the \emph{Wadge rank} of \( A \) (or of \( \wadgedegree{A}^X \)) and is denoted 
by \( \wrank{A}^{X} \). For technical and historical reasons,  
we use \( 1 \)  as the minimal value for 
\( \wrank{A}^{X} \). We also let \( \Theta_X \) be the length of \( \wadge{X} \), i.e.\ \( \Theta_X = \sup \{ \wrank{A}^X +1 \mid A \subseteq X \} \).

To simplify the notation, given a set \( A \subseteq X \) we sometimes write \( \neg A \) instead of \( X \setminus A \) when the ambient space \( X \) is clear from the context.
Wadge left his name to continuous reducibility when he proved in \cite{wadgephd}, by game theoretic methods, that $\lew^{\baire}$ on Borel subsets of the Baire space $\baire = {}^\omega \omega$ is well-founded and satisfies
the \emph{Semi-Linear Ordering principle}, or \( \mathsf{SLO}^{\w}_{\baire} (\mathrm{Bor}) \):
for all (Borel) sets $A,B \subseteq \baire$ 
\begin{equation} \label{eq:SLOW}
A \lew^{\baire} B \quad \text{or} \quad \neg B \lew^{\baire} A.
\end{equation}
These two results generalize under the \emph{Axiom of Determinacy} ($\AD$), yielding to the fact that the Wadge hierarchy is \emph{semi-well-ordered}, namely: the  full \( \mathsf{SLO}^{\w}_{\baire} \) holds (i.e.\ \eqref{eq:SLOW} is true for arbitrary subsets  \( A,B \subseteq \baire \)), and \( \lew^{\baire} \) is well-founded. (The former directly follows from Wadge's original lemma, while the latter is due to a highly nontrivial argument by Martin and Monk.) 

Early investigations on Wadge theory \cite{wadgephd,vanwesepphd,steelphd,solovay1978} have fully described
the Wadge hierarchy on \( \baire\) under
 $\AD$ (see Figure~\ref{fig:baire}). Indeed, by \( \mathsf{SLO}^{\w}_{\baire} \)  one easily gets that 
for each \( A \subseteq \baire \), either \( A \lew^{\baire} \neg A \), in 
which case we say that \( A \) (or its Wadge degree \( \wadgedegree{A}^{\baire} \)) 
is \emph{selfdual}, or else \( A \) and \( \neg A \) are \( \lew^{\baire} \)-incomparable, in which 
case \( A \) is called \emph{nonselfdual} and 
\( \{ \wadgedegree{A}^{\baire}, \wadgedegree{\neg A}^{\baire} \} \) is a 
maximal antichain, called a \emph{nonselfdual pair}. The minimal Wadge degrees 
are those in the nonselfdual pair 
\( \{ \wadgedegree{\baire}^{\baire}, \wadgedegree{\emptyset}^{\baire} \} \), and indeed \( \wadgedegree{\baire}^{\baire} = \{ \baire \} \) and 
\( \wadgedegree{\emptyset}^{\baire} = \{ \emptyset \} \). After this minimal nonselfdual pair there is 
a single selfdual degree consisting of all nontrivial clopen subsets of \( \baire \). Then nonselfdual pairs 
and selfdual degrees alternate, at limit levels of countable cofinality there is always a 
selfdual degree, while at limit levels of uncountable cofinality there is always a 
nonselfdual pair. Finally, \( \Theta_{\baire} = \Theta \) where
\[ 
\Theta = \sup \{ \alpha \in \On \mid \RR \text{ surjects onto } \alpha \}.
 \]

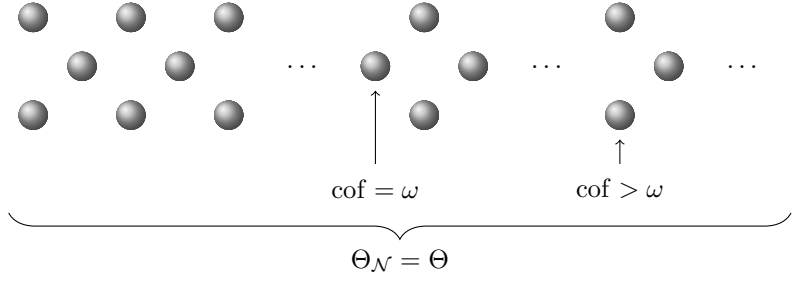
\begin{figure} 
\begin{tikzpicture}[scale=0.65]
\ball{0}{1};
\ball{0}{-1};
\ball{1}{0};
\ball{2}{1};
\ball{2}{-1};
\ball{3}{0};
\ball{4}{1};
\ball{4}{-1};
\node at (5.5,0) {\( \dotsc \)};
\ball{7}{0};
\ball{8}{1};
\ball{8}{-1};
\ball{9}{0};
\node at (10.5,0) {\( \dotsc \)};
\ball{12}{1};
\ball{12}{-1};
\ball{13}{0};
\node at (14.5,0) {\( \dotsc \)};
\node at (7,-2.5) {\( \mathrm{cof} = \omega \)};
\draw [->] (7,-2)--(7, -0.5);
\node at (12,-2.5) {\( \mathrm{cof} > \omega \)};
\draw [->] (12,-2)--(12, -1.5);
\draw [decorate,decoration={brace,mirror,amplitude=10pt}]
    (-0.5,-3) -- (15.5,-3) node [midway,yshift=-0.25in] {\( \Theta_{\baire} = \Theta \)};
\end{tikzpicture}
\caption{The Wadge hierarchy \( \wadge{\baire} \) on the Baire space \( \baire \).}
\label{fig:baire}
\end{figure}

Wadge theory has received a lot of attention, as witnessed by \cite{van_wesep,louveau,louveau_san_raymond,duparc1, andretta_martin,andrettaSLO, mottoros1,MotBeyond,MotBaire, kiharamontalban, carroy2020every}, among many other works.
Several generalizations have been considered, including variations of the reducibility (usually replacing continuous functions with other natural classes of functions) and/or replacing \( \baire \) with other spaces. In the second direction, the first natural move is to consider other zero-dimensional Polish spaces.%
\footnote{Recall that a space is \emph{Polish} if it is separable and completely metrizable, and \emph{zero-dimensional} if it has a basis of clopen sets.}
For example, if one assumes \( \AD\) again and considers the Cantor space \( \cantor = {}^\omega 2 \), then the corresponding Wadge hierarchy \( \wadge{\cantor} \) looks exactly like \( \wadge{\baire} \) except for the fact that at limit levels one always has a nonselfdual pair, independently of the cofinality of the level (see Figure~\ref{fig:cantor}). This is obtained via   straightforward modifications of the results and game-theoretic methods used in the case of \( \baire \).%
\footnote{Although this was probably known in the field since the Eighties, to the best of our knowledge the first complete accounts on the structure \( \wadge{\cantor} \) appeared only more recently in~\cite[Section 2.7]{andrettaSLO} and~\cite[Section 4]{AndCamCantorhierarchy}.}

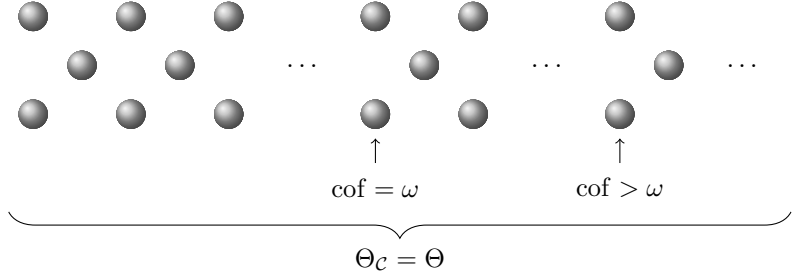
\begin{figure} 
\begin{tikzpicture}[scale=0.65]
\ball{0}{1};
\ball{0}{-1};
\ball{1}{0};
\ball{2}{1};
\ball{2}{-1};
\ball{3}{0};
\ball{4}{1};
\ball{4}{-1};
\node at (5.5,0) {\( \dotsc \)};
\ball{7}{-1};
\ball{7}{1};
\ball{8}{0};
\ball{9}{1};
\ball{9}{-1};
\node at (10.5,0) {\( \dotsc \)};
\ball{12}{1};
\ball{12}{-1};
\ball{13}{0};
\node at (14.5,0) {\( \dotsc \)};
\node at (7,-2.5) {\( \mathrm{cof} = \omega \)};
\draw [->] (7,-2)--(7, -1.5);
\node at (12,-2.5) {\( \mathrm{cof} > \omega \)};
\draw [->] (12,-2)--(12, -1.5);
\draw [decorate,decoration={brace,mirror,amplitude=10pt}]
    (-0.5,-3) -- (15.5,-3) node [midway,yshift=-0.25in] {\( \Theta_{\cantor} = \Theta \)};
\end{tikzpicture}
\caption{The Wadge hierarchy \( \wadge{\cantor} \) on the Cantor space \( \cantor \).}
\label{fig:cantor}
\end{figure}

If one moves to an arbitrary zero-dimensional Polish space \( X \) (equivalently: to an arbitrary closed subspace of \( \baire \)) different from \( \baire \) and \( \cantor \), instead, very little is known about how \( \wadge{X} \) looks like. The hierarchy starts again with the nonselfdual pair \( \{ \wadgedegree{X}^X, \wadgedegree{\emptyset}^X \} \), where \( \wadgedegree{X}^X = \{ X \} \) and \( \wadgedegree{\emptyset}^X = \{ \emptyset \} \), followed by the selfdual degree of all nontrivial clopen subsets of \( X \). 
As discussed in Section~\ref{subsec:retraction},
the existence of a retraction from \( \baire \) onto \( X \) yields the fact that, still under \( \AD  \), the Wadge hierarchy \( \wadge{X} \) is semi-well-ordered, i.e.\ the quasi-order \( \lew^X \) is well-founded and \( \mathsf{SLO}^\w_X \) is true, i.e.\ \( {A \lew^X B} \vee {\neg B \lew^X A}\) holds for all \( A,B \subseteq X \) (see also~\cite{mottorosquasiPolish}).
This implies that at each level of \( \wadge{X} \) we find either a nonselfdual pair \( \{ \wadgedegree{A}^X, \wadgedegree{\neg A}^X \} \) (when \( A \not\lew^X \neg A \)) or a single selfdual degree \( \wadgedegree{A}^X \) (when \( A \lew^X \neg A \)). 
Another easy observation is that \( \Theta_X = \Theta_{\baire} = \Theta \) if \( X \) is uncountable,
while \( \Theta_X < \omega_1 \) if \( X \) is countable. This follows from the retraction trick again (see Section~\ref{subsec:retraction}) if \( X \) is uncountable, because in this case \( X \) contains a closed homeomorphic copy of \( \cantor \); if instead \( X \) is countable, we use the fact that every subset of \( X \) is a \( \boldsymbol{\Delta}^0_2 \)-set, and indeed it is in \( \differenceopen{\alpha} \), where \( 1 \leq \alpha < \omega_1 \) depends on the Cantor-Bendixson rank \( \rkcb{X} \) of \( X \).
However, nothing else is known, and a complete and detailed description of \( \wadge{X} \)  has never been carried out in the literature: 
the main purpose of this article is to fill this gap and prove the following result. (Notice that in the countable case zero-dimensionality is automatic and could thus be removed from the hypotheses, and moreover in such case we do not need any determinacy assumptions because \( \pow(X) \subseteq \boldsymbol{\Delta}^0_2(X) \) and thus Martin's Borel determinacy is enough.)

\begin{maintheorem} \label{thm:maintheorem}
Let \( X \) be an arbitrary zero-dimensional Polish space, and assume \( \AD \)  if \( X \) is uncountable. Then 
\begin{enumerate-(1)}
\item \label{thm:maintheorem-1}
\( \wadge{X} \) is semi-well-ordered (\cref{prop:semiwellordered});
\item \label{thm:maintheorem-2}
selfdual degrees and nonselfdual pairs alternate, starting with the nonselfdual pair \( \{ \wadgedegree{X}^X, \wadgedegree{\emptyset}^X \} \) at the bottom, followed by the selfdual degree of all nontrivial clopen sets (\cref{after_nsd_is_sd}, \cref{alternating_duality_theorem_intro}, and \cref{prop:minimaldegree});
\item \label{thm:maintheorem-4}
for limit levels, we distinguish two cases:
\begin{enumerate-(a)}
\item \label{thm:maintheorem-4a}
if the perfect kernel of \( X \) is not compact (hence nonempty), then at all limit levels of uncountable cofinality there is a nonselfdual pair and  at all limit levels of countable cofinality there is a selfdual degree (\cref{uncountablecof_ordinal_intro} and \cref{wadgeZ_eq_wadgeBaire_onedir});
\item \label{thm:maintheorem-4b}
otherwise there is \( \alpha_X < \omega_1 \) such that for all limit \( \alpha < \Theta_X \), there is a selfdual degree at level \( \alpha \) if and only if \( \alpha < \alpha_X \);  in particular, \( \alpha_X = 0  \) if \( X \) is compact (\cref{countablecof_ordinal_intro} and the ensuing discussion);
\end{enumerate-(a)}
\item \label{thm:maintheorem-5}
\( \Theta_X = \Theta \) if \( X \) is uncountable, while if \( X \) is countable then $\Theta_X = 2 \cdot \rkcb{X} + \varepsilon_X$ for some $\varepsilon_X \in \{-1,0,1\}$ (\cref{prop:decompositionofslefdual} and \cref{length_countable_hierarchy_intro}).
\end{enumerate-(1)}
\end{maintheorem}

The parameter \( \alpha_X \) depends on a new rank on \( X \) which we call \emph{compact rank} and denote by \( \rkcomp{X}  \) (Section~\ref{subsec:compactrank}). 
As for item~\ref{thm:maintheorem-5}, the ordinal \( \rkcb{X} \) is the usual Cantor-Bendixson rank of \( X \), while the parameter \( \varepsilon_X \) depends on \( \rkcb{X} \), \( \alpha_X \), and whether \( X \) is simple  or not (see Section~\ref{subsec:CB-rank} for the relevant definitions). 

It follows from \cref{thm:maintheorem} that if \( X \) has a non-compact perfect kernel, then \( \wadge{X} \) is isomorphic to \( \wadge{\baire} \) (Figure~\ref{fig:baire}); if \( X \) is compact and, more generally, if \( \rkcomp{X} < \omega \), then \( \wadge{X} \) is isomorphic to \( \wadge{\cantor} \) (Figure~\ref{fig:cantor}); in the remaining cases, i.e.\ when \( \wadge{X} \) is isomorphic neither to \( \wadge{\baire}\) nor to \( \wadge{\cantor} \), the hierarchy \( \wadge{X} \) is a mixture of the two, with an initial segment (of countable length) which is Baire-like and the rest of the hierarchy which is Cantor-like (see Figure~\ref{fig:mixed}). 

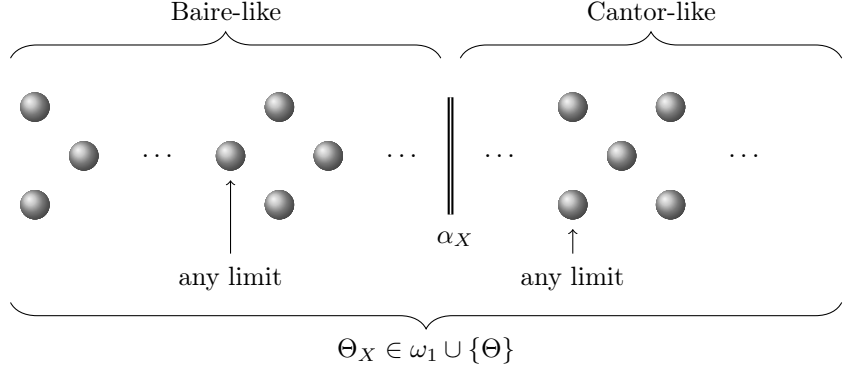
\begin{figure} 
\begin{tikzpicture}[scale=0.65]
\ball{0}{1};
\ball{0}{-1};
\ball{1}{0};
\node at (2.5,0) {\( \dotsc \)};
\ball{4}{0};
\ball{5}{1};
\ball{5}{-1};
\ball{6}{0};
\node at (7.5,0) {\( \dotsc \)};
\draw [thick,double] (8.5,1.2)--(8.5,-1.2);
\node at (9.5,0) {\( \dotsc \)};
\ball{11}{1};
\ball{11}{-1};
\ball{12}{0};
\ball{13}{1};
\ball{13}{-1};
\node at (14.5,0) {\( \dotsc \)};
\node at (4,-2.5) {any limit};
\draw [->] (4,-2)--(4, -0.5);
\node at (11,-2.5) {any limit};
\draw [->] (11,-2)--(11, -1.5);
\node at (8.6,-1.7) {\( \alpha_X \)};
\draw [decorate,decoration={brace,mirror,amplitude=10pt}]
    (-0.5,-3) -- (16.5,-3) node [midway,yshift=-0.25in] {\( \Theta_{X} \in \omega_1 \cup \{ \Theta \} \)};
    \draw [decorate,decoration={brace,amplitude=10pt}]
    (-0.5,2) -- (8.3,2) node [midway,yshift=0.25in] {Baire-like};
        \draw [decorate,decoration={brace,amplitude=10pt}]
    (8.7,2) -- (16.5,2) node [midway,yshift=0.25in] {Cantor-like};
\end{tikzpicture}
\caption{The Wadge hierarchy \( \wadge{X} \) on an arbitrary zero-dimensional Polish space \( X \) when \( \wadge{X} \) is isomorphic neither to \( \wadge{\baire} \) nor to \( \wadge{\cantor} \). In this case, the ordinal \( \alpha_X \) satisfies \( \omega < \alpha_X < \omega_1 \).}
\label{fig:mixed}
\end{figure}

\cref{thm:maintheorem}
 is optimal, in the sense that for all possibilities which are coherent with the limitations contained therein, there is a zero-dimensional Polish space \( X \) whose Wadge hierarchy \( \wadge{X} \) realizes precisely that scenario (Section~\ref{section_sharpness}). Thus what we have obtained is a complete classification (up to order-isomorphism) of all possible Wadge hierarchies for zero-dimensional Polish spaces, where the complete invariants are given by pairs of ordinals \( (\alpha_X, \Theta_X ) \) which can easily be coded as countable ordinals (see the discussion at the beginning of Section~\ref{sec:mainthm2}).

Although our classification of the possible Wadge hierarchies over zero-dimensional Polish spaces \( X \) is complete and quite satisfactory, we also show that it is not possible to determine \( \wadge{X} \) in a Borel manner. For example, by \cref{thm:complexityofeqclasses} the collection of all zero-dimensional Polish spaces whose Wadge hierarchy is order-isomorphic to \( \wadge{\baire} \) is a complete analytic subset of the Effros Borel space \( F(\baire) \). More generally, for all nontrivial zero-dimensional Polish spaces \( X \)
there is no Borel procedure to determine which zero-dimensional Polish spaces \( Y \) give rise to a Wadge hierarchy \( \wadge{Y} \) order-isomorphic to \( \wadge{X} \).

\begin{maintheorem} \label{thm:maintheorem2}
Let \( X \) be a zero-dimensional Polish space with at least two points, and assume \( \AD \) if \( X \) is uncountable.
Then the set of zero-dimensional Polish spaces \(Y\) such that \(\wadge{Y}\) is order-isomorphic to \(\wadge{X}\) is not Borel in the Effros Borel space \( F(\baire)\).
\end{maintheorem}

Indeed, it is not even possible to decide in a Borel way whether at a given countable limit level we will find a nonselfdual pair or a selfdual degree, or to determine the value of \( \alpha_X \) and/or \( \Theta_X \) (see \cref{rmk:complexityofeqclasses}). The reason behind this wild behaviour actually relies on the possible configurations of isolated points. If we restrict to perfect spaces, then the situation is radically neater and we have only two (Borel) possibilities: spaces with a Baire-like Wadge hierarchy, and spaces with a Cantor-like Wadge hierarchy (\cref{prop:complexityofeqclasses}).

Finally, let us briefly discuss the optimality of the various hypotheses in our results.
As it is well-known in the area, the determinacy assumption \( \AD \) can be relaxed to get level-by-level statements,
at the cost of making such statements and the various definitions a lot more cumbersome.
For a detailed discussion on this matter, and specifically on nice pointclasses, see \cite[Section 3]{articolo_raphael}. In particular, if we restrict the attention to the Borel subsets of \( X \), then Martin's Borel determinacy suffices and we do not need to assume \( \AD \) anywhere.

Moreover, we point out that there is a reason for restricting to zero-dimensional Polish spaces. The Wadge quasi-order has been studied in various other contexts, even beyond the realm of Polish spaces
(see for instance \cite{Selivanov,mottorosquasiPolish,pequignot,ikegamischlicht,camerlo,camerlomassaza}),
but in most cases it is not semi-well-ordered if the space is not zero-dimensional (with a few exceptions: see~\cite[Theorem 33]{camerlomassaza}).
If one restricts the attention to Polish spaces, this becomes even a characterization because  \( \wadge{X} \) has infinite antichains whenever \( X \) is a metric (not necessarily complete or separable) space with positive dimension (see~\cite[Theorem 1.5]{schlicht_nozerodim}). Moreover, 
\cite{hertling} and~\cite{duparc2} show among other things that \(\wadge{X} \) is ill-founded when \( X \) is $\RR$ or the Scott domain, respectively. See~\cite{mottorosquasiPolish} for more possible behaviours and a rough classification of Wadge-like hierarchies on arbitrary (quasi-)Polish spaces.

\section{Preliminaries and known facts}

The monograph \cite{kechris} is our reference for all classical results and notation in descriptive set theory.
From this point onward, unless otherwise stated, we assume \( \AD \) and that \( X \), \( Y \), \( Z \), \dots are nonempty zero-dimensional Polish spaces. Recall however
that the determinacy assumption is just for ease of exposition, and that no further assumption is required if all results and definitions are restricted to the realm of Borel sets.

\subsection{Some easy and well-known facts}

The principle \( \mathsf{SLO}^\w_X \) can be reformulated to deal with sets in possibly different zero-dimensional Polish spaces (\cite[Theorem 21.14]{kechris}). The proof is identical to the original one by Wadge, except that the rules of the so-called Wadge's game are modified so that the players stay inside their respective closed subspaces of \( \baire \).

\begin{proposition}[\( \AD \)] \label{prop:generalWadgelemma}
Let \( X,Y \) be zero-dimensional Polish spaces. Then for all $A\subseteq X$ and $B\subseteq Y$ 
\[
A\lew^{X,Y}B \quad \vee \quad \neg B\lew^{Y,X}  A.
\]
\end{proposition}

Notice that when \( X = Y \), \cref{prop:generalWadgelemma} is simply the assertion that \( \mathsf{SLO}^\w_X \) holds under \( \AD \) for an arbitrary zero-dimensional Polish space \( X \).
Proposition~\ref{prop:generalWadgelemma} is so crucial in the theory that we will simply refer to it as \emph{Wadge's Lemma}.

Let \( \sds{X} \) be the collection of all \emph{selfdual sets} of \( X \), i.e.\ 
\[ 
\sds{X} = \{ A \subseteq X \mid A \lew^X \neg A \},
 \] 
and, dually, let \( \nsds{X} \) be the collection of \emph{nonselfdual sets} of \( X \), that is
\[ 
\nsds{X} = \pow(X) \setminus \sds{X} = \{ A \subseteq X \mid A \not\lew^X \neg A \}.
 \] 
Then Wadge's Lemma implies that at each level of \( \wadge{X} \) we find either a \( 2 \)-sized maximal antichain \( \{ \wadgedegree{A}^X, \wadgedegree{\neg A}^X \} \) for some \( A \in \nsds{A} \), called a \emph{nonselfdual pair}, or a single selfdual degree \( \wadgedegree{A}^X \) for some \( A \in \sds{X} \). 

Another easy observation (see \cite[Lemma 1]{camerlo}) is that for every nonempty topological space \( X \) the pair \( \{ \wadgedegree{X}^X, \wadgedegree{\emptyset}^X \} \) is nonselfdual and its elements are minimal in \( \wadge{X} \); if \( X \) is not connected then right after such pair there is a single selfdual degree. (The proof of the next proposition is left to the reader.)

\begin{proposition} \label{prop:minimaldegree}
Let \( X \) be an arbitrary nonempty topological space. Then \( \wadgedegree{X}^X = \{ X \} \) and \( \wadgedegree{\emptyset}^X = \{ \emptyset \}\). Moreover, the pair \( \{ \wadgedegree{X}^X, \wadgedegree{\emptyset}^X \} \) is nonselfdual and \( \lew^X \)-minimal in \( \wadge{X} \), that is, for every  \( \emptyset, X \neq A \subseteq X\) we have \( \emptyset, X \slew^X A \).

If moreover \( X \) is disconnected, then immediately above \( \{ \wadgedegree{X}^X, \wadgedegree{\emptyset}^X \} \) there is a single selfdual degree consisting of all nontrivial clopen subsets of \( X \), and all other Wadge degrees are strictly \( \lew^X \)-above it.
\end{proposition}

\subsection{The retraction method} \label{subsec:retraction}

Part~\ref{thm:maintheorem-1} of \cref{thm:maintheorem} can be proved using a trick due to Marcone (unpublished, but see~\cite[Proposition 28]{andrettaSLO}) 
and, independently, to Selivanov~\cite[Corollary 2.4]{Selivanov}.  
If \( X \subseteq Y \) are such that there is a  retraction%
\footnote{A retraction from \( Y \) onto \( X \) is a (surjective) continuous function \( \rho \colon Y \to X \) such that \( \rho \restriction X \) is the identity on \( X \).}
\( \rho \) from \( Y \) onto \( X \), then 
 the map from \( \pow(X) \) to \( \pow(Y) \) defined by \( A \mapsto \rho^{-1}(A) \) embeds the structure \( \langle \pow(X), {\lew^X}, {\neg} \rangle \) into \( \langle \pow(Y), {\lew^{Y}}, {\neg} \rangle \) 
 (see also \cite[Proposition 5.4]{mottorosquasiPolish} for a more general result). 
 Notice that up to Wadge equivalence, the choice of \( \rho \) is irrelevant: if \( \rho_1, \rho_2 \colon Y \to X \) are two retractions then \( \rho_1^{-1}(A) \eqw^Y \rho_2^{-1}(A) \), as witnessed by the maps \( \rho_1 \) and \( \rho_2 \) themselves. Notice also that for any retraction \( \rho \colon Y \to X \) and \( A \subseteq X \), we have \( A \eqw^{X,Y} \rho^{-1}(A) \) via the inclusion map \( \iota \colon X \to Y \) and the retraction \( \rho \colon Y \to X \).

It is a classical result that every zero-dimensional Polish space can be construed as a closed subspace of \( \baire \), thus there is a retraction of \( \baire \) onto \( X \).
Therefore we get the desired result (\cite[Theorem 5.12(1)]{mottorosquasiPolish}).

\begin{proposition}[\( \AD \)] \label{prop:semiwellordered}
Let \( X \) be a zero-dimensional Polish space. Then \( \mathsf{SLO}^\w_X \) holds and \( \lew^X \) is well-founded, i.e.\ \( \wadge{X} \) is semi-well-ordered.
\end{proposition}

\cref{prop:semiwellordered} allows us to split ordinals \( 1 \leq \alpha < \Theta_X \) in two groups: those for which at level \( \alpha \) of \( \wadge{X} \) we find a single selfdual degree and those for which we instead find a nonselfdual pair. This is condensed in the following definition and notation.

\begin{defin} \label{def:sdordinal}
An ordinal $1 \leq \alpha < \Theta_X$ is called \emph{selfdual} for $X$ if the subsets of \( X \) with Wadge rank $\alpha$ are selfdual,
	otherwise we say that $\alpha$ is \emph{nonselfdual}. We denote by $ \sdordinal{X}$ the set of all \( 1 \leq \alpha < \Theta_X \) which are selfdual for $X$,
	while $ \nsdordinal{X}$ stands for the set of nonselfdual \( 1 \leq \alpha < \Theta_X \). \end{defin}

Notice that by \cref{prop:semiwellordered}, determining the structure of \( \wadge{X} \) amounts to determine the two sets of ordinals \( \sdordinal{X} \) and \( \nsdordinal{X} \), together with the exact value of \( \Theta_X \).

The same trick based on retractions can also be used to prove the first half of part~\ref{thm:maintheorem-5} of \cref{thm:maintheorem}, i.e.\ to compute
\( \Theta_X \) when \( X \) is uncountable. Indeed, if a zero-dimensional Polish space \( X \) is uncountable, then \( \cantor \) is (homeomorphic to) a closed subset of \( X \): therefore there is a retraction of \( X \) onto \( \cantor \), and thus \( \langle \pow(\cantor), {\lew^\cantor}, {\neg} \rangle \) embeds into \( \langle \pow(X), {\lew^X}, {\neg} \rangle \). Combined with the fact that, as observed, \( \langle \pow(X), {\lew^X}, {\neg} \rangle \) embeds into \( \langle \pow(\baire), {\lew^\baire}, {\neg} \rangle \), this shows that \( \Theta_\cantor \leq \Theta_X \leq \Theta_\baire \). 
Since under \( \AD \) it holds  \( \Theta_\cantor = \Theta_\baire =  \Theta \), we get:  

\begin{proposition}[\( \AD \)] \label{prop:decompositionofslefdual}
Let \( X \) be an uncountable zero-dimensional Polish space. Then \( \wadge{X} \) has length \( \Theta_X = \Theta \).
\end{proposition}

Finally, combining the retraction method with a well-known result in Wadge theory due to Wadge himself~\cite[Lemma 3]{AndrettaMoreOnWadge},  we obtain a technical fact that provides a useful characterization of selfdual sets.

\begin{proposition}[\( \AD \)] \label{prop:selfdualanalysis}
Let \( X \) be a zero-dimensional Polish space. Then \( A \subseteq X \) is selfdual if and only if there are a clopen partition \(  (V_n)_{n \in \omega } \) of \( X \) and sets \( (A_n)_{n \in \omega} \) such that \( A_n \slew^X A \) and \( A = \bigcup_{n \in \omega} (A_n \cap V_n) \).

Moreover, we can further assume that each \( A_n \) is nonselfdual, and \( \wrank{A}^X = \sup_{n \in \omega} (\wrank{A_n}^X+1) \).
\end{proposition}

\begin{proof}
The backward direction is obvious, as by \( \mathsf{SLO}^\w_X\)  we have \( A_n \lew^X \neg A \) for all \( n \in \omega \): if \( (f_n)_{n \in \omega} \) are continuous maps witnessing this, then \( f = \bigcup_{n \in \omega} (f_n \restriction V_n) \) is a continuous map witnessing \( A \lew^X \neg A \).

Conversely, let \( X \) be a closed subspace of \( \baire \) and \( \rho \colon \baire \to X \) be a retraction. If \( A \) is selfdual in \( X \), then \( A' = \rho^{-1}(A) \) is selfdual in \( \baire \). Then the set of those \( s \in \pre{<\omega}{\omega} \) such that \( A' \lew^{\baire,\Nbhd_s} A' \cap \Nbhd_s \) is a nonempty well-founded tree \( \boldsymbol{T}(A') \subseteq \pre{<\omega}{\omega} \) (\cite[Lemma 22]{AndrettaEquivalenceWadgeLipschitz}). Let \( (V_n)_{n \in \omega} \) be an enumeration of the sets of the form \( \Nbhd_{s {}^\smallfrown{} k} \cap X \) with \( k \in \omega \) and \( s \) a leaf in \( \boldsymbol{T}(A') \), and set \( A_n = X \) if \( V_n \subseteq A \) and \( A_n =  A \cap V_n \) otherwise.
It is easy to check that the \( V_n \)'s and \( A_n \)'s are as required. 

The additional part on nonselfduality of the \( A_n \)'s follows by iterating the construction on those \( A_n \) which happen to be selfdual in \( X \). (This process must terminate  on each branch after finitely many steps because otherwise we would construct a strictly \( \lew^X \)-decreasing sequence of subsets of \( X \), against~\cref{prop:semiwellordered}.)

Finally, set \( \alpha = \wrank{A}^X \) and \( \alpha_n = \wrank{A_n}^X \), so that in particular \( \alpha_n < \alpha \).
If \( B \subseteq X \) is a set with \( \wrank{B}^X = \sup_{n \in \omega} (\alpha_n+1) \), then \( A_n \slew^X B \) because \( \wrank{A_n}^X = \alpha_n < \alpha_n+1 \leq  \wrank{B}^X \): if \( f_n \) witnesses \( A_n \lew^X B \), then \( A \lew^X B \) via \( \bigcup_{n \in \omega} (f_n \restriction V_n) \), thus \( \sup_{n \in \omega} (\alpha_n+1) = \wrank{B}^X \geq \wrank{A}^X = \alpha \). The other inequality is obvious.
\end{proof}

\begin{remark} \label{rmk:sefldualanalysis}
\begin{enumerate-(1)}
\item \label{rmk:sefldualanalysis-1}
The above proof shows that we can indeed assume that either \( A_n = X \) (if \( V_n \subseteq A \)), or else \( A_n = A \cap V_n \) (and thus \( A_n \subseteq V_n \)); indeed, if \(A \) is not clopen and we drop the requirement that \(A_n \) be nonselfdual, then the latter can be assumed to be true for all \( n \in \omega \).
\item \label{rmk:sefldualanalysis-3}
Proposition~\ref{prop:selfdualanalysis} can be relativized to any clopen subset of \( X \), namely:
If $A\subseteq X$ is selfdual and $U\in\clopensets{X}$, then there exist a clopen partition $(U_n)_{n \in \omega}$ of \( U \) and nonselfdual sets $A_n\slew^X A$
such that $\bigcup_{n \in \omega} (A_n\cap U_n)=A\cap U$. Indeed, this is trivial if \( U \cap A \slew^X A \). If instead \( U \cap A \equiv_\w^X A \), then \( U \cap A \) is selfdual and we can apply \cref{prop:decompositionofslefdual}, setting then \( U_n = V_n \cap U \). This easily provides alternative proofs of~\cite[Theorem 5.4 and Corollary 5.5]{articolo_raphael}.
\end{enumerate-(1)}
\end{remark}

\subsection{Pointclasses}

Wadge pointclasses provide an alternative (but equivalent) way to present and study  the Wadge hierarchy on a space \( X \). 
A \emph{boldface pointclass} \( \nsdclass \) in $X$ is a subset of $\pow(X)$ which is closed under continuous preimages, i.e.\ it is downward closed under \( \lew^X \).
The \emph{dual} \( \widecheck{\nsdclass} \) of \( \nsdclass \) is defined by \( \widecheck{\nsdclass}=\{\neg A \mid A\in\nsdclass \} \). The dual operator is clearly idempotent, i.e.
 the dual of \( \widecheck{\nsdclass} \) is \( \nsdclass \) itself. We say that \( \nsdclass \) is \emph{nonselfdual} if \( \nsdclass \neq \widecheck{\nsdclass} \), and \emph{selfdual} otherwise.

 A boldface pointclass \( \nsdclass \) is a \emph{Wadge class} if it is of the form $\wadgeclass{A}_X=\{ B\subseteq X \mid B\lew^X A \}$ for some \( A \subseteq X \); any \( A \in \nsdclass \) satisfying \( \nsdclass = \wadgeclass{A}_X \) is called \emph{complete} (for \( \nsdclass \)), and we say that \( \nsdclass \) is \emph{generated} by \( A \). We denote by \( \wc{X} \) the collection of all Wadge classes in \( X \).
 Notice that by \( \mathsf{SLO}^\w_X \), every nonselfdual boldface pointclass \( \nsdclass \) of \( X \) is automatically a Wadge class; 
 in contrast, not all selfdual boldface pointclasses in \( X \) are Wadge classes.
 We denote by \( \nsdc{X} \) the collection of all nonselfdual Wadge classes in \( X \), and by \( \sdc{X} = \wc{X} \setminus \nsdc{X} \) the collection of the selfdual ones. 
 
The Wadge hierarchy \( \wadge{X} \)  is clearly isomorphic to the structure of all  Wadge classes ordered by inclusion, as witnessed by the map \( \wadgedegree{A}^X \mapsto \wadgeclass{A}_X \). In particular, \( A \in \sds{X} \)  if and only if \( \wadgeclass{A}_X \in \sdc{X} \), and a nonselfdual pair \( \{ \wadgedegree{A}^X, \wadgedegree{\neg A}^X \} \) corresponds to the pair of distinct nonselfdual Wadge classes \( (\nsdclass, \widecheck{\nsdclass} ) \) with \( \nsdclass = \wadgeclass{A}_X \).
Every $\nsdclass \in \wc{X}$ induces a \emph{coarse Wadge class} $\nsdclass^*= \nsdclass \cup\widecheck{\nsdclass}$.
When \( \mathsf{SLO}^\w_X \) holds and \( \lew^X \) is well-founded, 
inclusion on coarse Wadge classes is a well-order: the ordinal \( \alpha \) corresponding to the position of \( \nsdclass^* \) in such a well-order (once we start counting by \( 1 \)) coincides with \( \wrank{A}^X \) for some/any \( A \subseteq X \) such that \( \nsdclass = \wadgeclass{A}_X \).
Moreover, when \( X \subseteq Y \) are zero-dimensional Polish spaces with \( X \) closed in \( Y \),
the retraction method (Section~\ref{subsec:retraction}) induces an injection from \( \wc{X} \) into \( \wc{Y} \). Indeed, if \( \rho \colon Y \to X \) is a retraction we can associate to each \( \nsdclass = \wadgeclass{A}_X \in \wc{X} \) the  Wadge class \( \rho^{-1}(\nsdclass) =  \wadgeclass{(\rho^{-1}(A))}_Y \in \wc{Y} \). This is again independent on the chosen retraction, i.e.\ if \( \rho_1, \rho_2 \colon Y \to X \) are both retractions, then \( \rho_1^{-1}(\nsdclass) = \rho_2^{-1}(\nsdclass) \) because \( \rho_1^{-1}(A) \eqw^Y \rho_2^{-1}(A) \). Clearly, \( \nsdclass \in \nsdc{X} \iff \rho^{-1}(\nsdclass) \in \nsdc{Y} \).

Notice that any boldface pointclass \( \nsdclass \neq \{ \emptyset \} \) completely determines the space \( X \) where it ``lives'', as \( X \) can be canonically recovered as the \( \subseteq \)-largest set in \( \nsdclass \). This allows us to greatly simplify the notation: the symbol \( \w \) is unnecessary%
\footnote{The link to continuous reducibility is somewhat implicit in the fact that we only consider \emph{boldface} pointclasses, which by definition are \( \lew^X \)-downward closed.}
because the order relation for pointclasses is simply inclusion; and as observed we can (almost) always drop the reference to the ambient space \( X \). So for example when \( \nsdclass \neq \{ \emptyset \} \) is a Wadge class we can simply write \( \rank{\nsdclass} \) to denote the Wadge rank of \( \nsdclass \), i.e.\ the ordinal \( \wrank{A}^X \) for some/any \( A \subseteq X \) such that \( \nsdclass = \wadgeclass{A}_X \).

The chosen notation highlights the tight connection between ordinals, sets, and Wadge classes in a zero-dimensional Polish space \( X \): if \( \nsdclass \) is a Wadge class generated by \( A \subseteq X \),  then
\[ 
\nsdclass \in \nsdc{X} \iff A \in \nsds{X} \iff \rank{\nsdclass} = \wrank{A}^X \in \nsdordinal{X} .
 \]

\subsection{The relativization method}

Building on previous work  by Louveau and Saint-Raymond, 
the paper \cite{articolo_raphael} develops a substantial part of Wadge theory for arbitrary zero-dimensional Polish spaces, often in terms of Wadge classes. Here we recall the few facts that we need for the present work.

The relativization method, introduced in~\cite{louveau_san_raymond}, provides a way to ``transfer'' boldface pointclasses in \( \baire \) to the space \( X \)   (see \cite[Section 6]{articolo_raphael}). 

\begin{defin} \label{def:relativization}
Let $X$ be a zero-dimensional Polish space,  and consider any boldface pointclass $\nsdclass\subseteq\pow(\baire)$. The \emph{relativization} of \( \nsdclass \) to \( X \) is the boldface pointclass
\[
\nsdclass(X)=\{A \subseteq X \mid g^{-1}(A)\in\nsdclass \text{ for all continuous } g \colon \baire \to X \}.
\]
\end{defin}

Notice that the dual (in \( X \)) of \( \nsdclass(X) \) is simply \( \widecheck{\nsdclass}(X) \), i.e.\ the relativization of the dual (in \( \baire \)) of \( \nsdclass \). Thus if \( \nsdclass(X) \in \nsdc{X} \) then \( \nsdclass \in \nsdc{\baire} \). In contrast, there are situations where \( \nsdclass \in \nsdc{\baire} \) but \( \nsdclass(X) \in \sdc{X} \), see the end of this subsection for more details.

The relativization method works particularly well with nonselfdual  Wadge classes.
The following result sums up the content of \cite[(Proof of) Lemma 6.5 and Theorem 7.2]{articolo_raphael}.

\begin{theorem}[\( \AD \)]\label{relativization} 
Let $X$  be a zero-dimensional Polish space.
\begin{enumerate-(1)}
\item  \label{relativization-2}
For all $\mathbf{\Lambda}\in\nsdc{X}$ there is a unique boldface pointclass $\nsdclass \subseteq \pow(\baire)$ such that $\mathbf{\Lambda}=\nsdclass(X)$, and moreover \( \nsdclass\in\nsdc{\baire} \).
\item  \label{relativization-3}
If moreover $X$ is uncountable, then
$\nsdc{X}=\set{\nsdclass(X)}[\nsdclass\in\nsdc{\baire}]$, and indeed \( \langle \nsdc{\baire}, {\subseteq} \rangle \) and \( \langle \nsdc{X}, {\subseteq} \rangle \) are isomorphic via the map \( \nsdclass \mapsto \nsdclass(X) \). 
\end{enumerate-(1)}
\end{theorem}

\cref{relativization}\ref{relativization-3} relies on the fact that if \( \nsdclass \in \nsdc{\baire} \) then \( \nsdclass(X) \) is still nonselfdual if \( X \) is uncountable. In countable spaces this ceases to be true: since all sets in a countable Polish space $X$ are $\boldsymbol{\Delta}^0_2$,
all pointclasses in $\nsdc{\baire}$ containing $\boldsymbol{\Delta}^0_2$ relativize to $\pow(X)$, which is selfdual. We will provide the correct analogue of \cref{relativization}\ref{relativization-3} for countable Polish spaces in \cref{cor:diffclassesinuncountable}.

\subsection{Cantor-Bendixson derivatives and rank} \label{subsec:CB-rank}

Recall from~\cite[Section 6.C]{kechris} the classical notion of \textit{Cantor-Bendixson derivative}, defined by
\[
D_{\cb}(Y) =\{x \in Y \mid x \text{ is not isolated in $Y$}\},
\]
and its iterates \( \derrkcb{\alpha}{X} \), recursively defined by setting \( \derrkcb{0}{X} = X \), \( \derrkcb{\alpha+1}{X} = D_{\cb}(\derrkcb{\alpha}{X}) \), and \( \derrkcb{\lambda}{X} = \bigcap_{\alpha < \lambda} \derrkcb{\alpha}{X} \) for \(\lambda\) limit. When \( X \) is Polish, there is \( \alpha < \omega_1 \) such that \( \derrkcb{\alpha}{X} = \derrkcb{\alpha+1}{X} \): the smallest such \( \alpha \) is the Cantor-Bendixson rank (\emph{CB-rank} for short) of \( X \) and is denoted by \( \rkcb{X}  \), while  $\kerp{X} = \derrkcb{\rkcb{X}}{X}$ is called \emph{perfect kernel} of \( X \) and can equivalently be defined as the largest perfect%
\footnote{A subspace of \( X \) is perfect if it is closed and has no isolated point.}
subspace of \( X \). 
 Moreover, for $x \in X \setminus \kerp{X}$ we let $\rkcbx{x}{X}$ be the unique ordinal $\alpha<\rkcb{X}$
satisfying $x \in \derrkcb{\alpha}{X} \setminus  \derrkcb{\alpha+1}{X}$.

If \( X \) is countable, then \( \kerp{X} = \emptyset \). In this case, we define the \emph{Cantor-Bendixson degree} \( \mathrm{Deg}_{\cb}(X) \) of $X$ (\emph{CB-degree} for short) as
the cardinality of $\derrkcb{\alpha}{X}$ if $\rkcb{X}=\alpha+1$ is a successor ordinal, and $\mathrm{Deg}_{\cb}(X)=\omega$  otherwise. The \emph{Cantor-Bendixson type} of $X$ (\emph{CB-type} for short)  is the pair \( \cbtype(X) = ( \rkcb{X} , \mathrm{Deg}_{\cb}(X)) \).
For compact countable Polish spaces, \( \cbtype(X) \) is a complete invariant for homeomorphism.

We say that a countable Polish space \( X \)  is \emph{simple} if it has successor CB-rank $\rank{X}_{\cb} = \alpha+1$ and  \( \mathrm{Deg}_{\cb}(X) = 1 \),
that is, if $\derrkcb{\alpha}{X}$ is a singleton.
The following is a special case of \cite[Lemma 2.4]{jslcarroy}, and shows in particular
 that if \( X \) is a countable Polish space with limit CB-rank, then \( X \) is not compact.

\begin{lemma}\label{decomposition_in_simple_clopen}
   Every countable Polish space $X$ of CB-type $(\alpha,\beta) \in \omega_1\times(\omega+1)$ admits a partition in nonempty clopen subsets
    $(C_n)_{n \in \beta}$ such that  each $C_n$ is simple,
    $\rkcb{C_n}=\alpha$ if $\alpha$ is successor, and $\rkcb{C_n}<\alpha$ with \( \sup_{n \in \omega} \rkcb{C_n} = \alpha \) if $\alpha$ is limit. 
\end{lemma}

The following is instead a direct application of \cite[Corollary 2.3]{jslcarroy}.

\begin{lemma}\label{clopen_rkgb_no_change}
	Let $X$ be a zero-dimensional  Polish space and $U \subseteq X$ be an open set such that \( U \cap \kerp{X} = \emptyset \).
	Then $\rkcb{U}= \min \set{ \alpha \leq \rkcb{X}}[ U \cap \derrkcb{\alpha}{X}= \emptyset]$. 
\end{lemma}

\subsection{The pointed gluing operation} 

We adopt a standard notation for sequences: for example, \( \lh(s) \) denotes the length of \( s \),  $s \conc t$ stands for the concatenation of $s$ and $t$, while for \( i \in \omega \) and \( n \in \omega  \) we denote by \( (i)^n \) (respectively, \( \inftycostbaire{i} \)) the constant sequence of length \( n \) (respectively, the \( \omega \)-sequence) with value \( i \). If $X$ is a topological space, we say that a sequence $(A_n)_{n \in \omega}$ of subsets of \( X \) converges to a point \( x \in X \), in symbols $A_n \to x$, if any open neighborhood of $x$ contains all but finitely many $A_n$'s.

The following construction, taken from~\cite[Section 3]{jslcarroy}, is implicitly used in the area since Wadge's thesis~\cite{wadgephd} and has been considered, sometimes with slight variations, by several authors (see e.g.\ \cite[Section 5]{Selivanovnew}).
	Given a sequence $(A_n)_{n \in \omega}$ with $A_n \subseteq \baire$, its \emph{pointed gluing} is the set
	\begin{equation*}
		\ptglsq{(A_n)_{n \in \omega}} = \inftycostbaire{0} \cup \bigcup_{n \in \omega}\{{(0)^{n}}\conc(1)\conc x \mid x \in A_n\}
	\end{equation*}
	Intuitively, \( \ptglsq{(A_n)_{n \in \omega}} \) is thus obtained by taking copies  of the \( A_n \)'s which converge to \(  \inftycostbaire{0} \), together with such limit.
	The functional counterpart of the pointed gluing operation is defined by considering a sequence \( (f_n)_{n \in \omega} \) of functions \( f_n \colon A_n \to B_n \), and letting 
	\[ 
	\ptglsq{(f_n)_{n \in \omega}} \colon \ptglsq{(A_n)_{n \in \omega}} \to \ptglsq{(B_n)_{n \in \omega}} 
	\] 
	be defined as the map sending \( \inftycostbaire{0} \) to itself and points of the form \( (0)^n\conc(1)\conc x \) with \( x \in A_n \) to \( (0)^n\conc(1)\conc f_n(x) \). 
When $A_n = A $ or \( f_n = f \) for all $n \in \omega$, to simplify the notation we write $\ptglsq{A}$ instead of $\ptglsq{(A_n)_{n \in \omega}}$ or \( \ptglsq{f} \) instead of \( \ptglsq{(f_n)_{n \in \omega}} \), respectively.

It is easy to verify that if  $A_n$ is closed (respectively, compact), then so is $\ptglsq{(A_n)_{n \in \omega}}$, and that $\ptglsq{(f_n)_{n \in \omega}}$ is continuous (respectively, injective, an embedding, or a homeomorphism) if so are the \( f_n \)'s. Moreover, if the sequence of ordinals $(\rkcb{A_n})_{n\in\omega}$ is non-decreasing, then an easy computation shows that $\rkcb{\ptglsq{(A_n)_{n \in \omega}}}=(\sup_{n \in \omega}\rkcb{A_n})+1$ (see e.g.\ \cite[Proposition 3.1]{jslcarroy}).

\subsection{Difference hierarchy} \label{subsec:diffhierarchy}

We recall the definition of difference hierarchy (see e.g.\ \cite[Section 22.E]{kechris}). Let $X$ be a Polish space, $\alpha \geq 1$ be a countable ordinal, and $(A_{\gamma})_{\gamma<\alpha}$ be a sequence of subsets of \( X \). Then we define
\[
	D_{\alpha}((A_{\gamma})_{\gamma < \alpha}) =
		\bigcup \bigg \{A_{\gamma} \setminus \bigcup_{\gamma' < \gamma} A_{\gamma'} \mid \gamma < \alpha \text{ and } \gamma \text{ has parity opposite to }\alpha\bigg\}.
\]

Moreover, when the boldface pointclass  $\nsdclass\subseteq\pow(X)$ is different from \( \{ \emptyset \} \) we set $
\difference{\alpha}{\nsdclass} = \{\differenceseq{\alpha}{(A_{\gamma})_{\gamma < \alpha}} \mid A_{\gamma}\in \nsdclass\}$. By convention, we set $\differenceseq{0}{\nsdclass} = 
\set{\emptyset}$. Notice also that \( \differenceseq{1}{\nsdclass} = \nsdclass \).

We will always work with \( \nsdclass = \boldsymbol{\Sigma}^0_1(X) \). One can easily verify that \( \differenceopenX{\alpha}{X}  \) is a boldface pointclass (in \( X \)), and that \( \differenceopenX{\alpha}{X}  = \differenceopen{\alpha}(X) \). 
Thus 
we can unambiguously 
denote by 
\( \differenceopenXcheck{\alpha}{X} \)
the dual of \( \differenceopenX{\alpha}{X} \).

Wadge himself proved in~\cite{wadgephd} that the first \( \omega_1 \)-many nonselfdual levels of \( \wadge{\baire} \) are occupied by the difference classes \( \differenceopen{\alpha} \), \( \alpha < \omega_1 \), and their duals.
It  follows that \( \rank{\differenceopen{\alpha}}_{\baire} = 2 \alpha+1 \), and that using the well-known successor and countable sum operations from~\cite[Chapter 3]{wadgephd} one can define complete sets (up to Wadge equivalence) for such classes. More precisely, we recursively  define a sequence of Wadge degrees \( (\mathbf{D}_\alpha)_{\alpha< \omega_1} \) by setting
\begin{equation} \label{eq:differencedegrees}
\mathbf{D}_\alpha = 
\begin{cases}
  \wadgedegree{\emptyset}^\baire & \text{if }\alpha = 0 \\
 \wadgedegree{ \{ (0)^n \conc (m+1) \conc x \mid n,m \in \omega \wedge x \in \neg B_\beta \}}^\baire & \text{if }\alpha = \beta+1 \\ 
 \wadgedegree{ \{ (0)^n \conc (m+1) \conc x \mid n,m \in \omega \wedge x \in \neg B_{\alpha_n} \}}^\baire & \text{if } \alpha \text{ is limit,}
\end{cases}
\end{equation}
where \( B_\beta \) is any set in \( \mathbf{D}_\beta \), \( (\alpha_n)_{n \in \omega}\) is some/any increasing sequence cofinal in \(\alpha\), and \( B_{\alpha_n} \) is an arbitrary set in \( \mathbf{D}_{\alpha_n} \). An easy induction shows that the sequence \( (\mathbf{D}_\alpha)_{\alpha < \omega_1} \) is well-defined, i.e.\ that up to \( \eqw^\baire \) it is independent of the choice of \( B_\beta \), \((\alpha_n)_{n \in \omega} \), and \( B_{\alpha_n} \). Moreover, one can prove by induction again that if \( B_\alpha \in  \mathbf{D}_\alpha \) then \( \differenceopen{\alpha} = \wadgeclass{(B_\alpha)}_\baire \).

\section{Our toolbox}

\subsection{More on the retraction method}

The retraction method immediately settles our problem for zero-dimensional Polish spaces which are not \emph{\( \sigma \)-compact}, i.e.\ cannot be written as a countable union of compact spaces, and for those which are uncountable and compact. 
This shows that we could in principle concentrate on \( \sigma \)-compact non-compact spaces, although our analysis will be more general and work for all zero-dimensional Polish spaces, without those limitations.

The next two results do not require any extra assumption beyond our base theory \( \ZF + \DC(\RR) \).

\begin{lemma}  \label{lem:biembimpliessamewadgehierarchy}
Suppose that \( X,Y \) are zero-dimensional Polish spaces, each of which is homeomorphic to closed subset of the other one. Then \( \wadge{X} \) and \( \wadge{Y} \) are isomorphic.
\end{lemma}

\begin{proof}
Without loss of generality, \( X \) is a closed subset of \( Y \) and \( Y \) is a closed subspace of some \( X ' \) which is homeomorphic to \( X \) via some \( h \colon X \to X' \). Let \( \rho_1 \colon Y \to X \) and \( \rho_2 \colon X' \to Y \) be retractions. We already observed that \( A \mapsto \rho_1^{-1} (A) \) provides an embedding of \( \langle \pow(X), {\lew^X}, {\neg} \rangle \) into \( \langle \pow(Y), {\lew^Y}, {\neg} \rangle \), and hence it canonically induces an embedding of \( \wadge{X} \) into \( \wadge{Y} \): we claim that the latter is also surjective (hence an isomorphism), i.e.\ that for every \( B \subseteq Y \) there is \( A \subseteq X \) such that \( B \eqw^Y \rho_1^{-1}(A) \). To see this, it is enough to notice that since \( \rho_1 \) and \( \rho_2 \) are retractions and \( h \) is a homeomorphism, then
\[ 
B \eqw^{Y,X'} \rho_2^{-1}(B) \eqw^{X',X} h^{-1}(\rho_2^{-1}(B)) \eqw^{X,Y} \rho_1^{-1}(h^{-1}(\rho_2^{-1}(B))).
 \] 
Thus setting \( A = h^{-1}(\rho_2^{-1}(B)) \) we are done.
\end{proof}

Combining the above lemma with the universality properties of \( \baire \) and \( \cantor \) (\cite[Theorem 7.8]{kechris}), the Cantor-Bendixson theorem (\cite[Corollary 6.5]{kechris}) and the Hurewicz theorem (\cite[Theorem 7.10]{kechris}) we obtain the following.

\begin{proposition} \label{prop:isomorphicWadgethroughretractions} 
Let \( X \) be a zero-dimensional Polish space.
\begin{enumerate-(1)}
\item \label{prop:isomorphicWadgethroughretractions-1}
If  \( X \) is not \(\sigma\)-compact,%
\footnote{Notice that this hypothesis implies in particular that \( X \) is uncountable.}
then \( \wadge{X} \) is isomorphic to \( \wadge{\baire} \).
\item \label{prop:isomorphicWadgethroughretractions-2}
If \( X \) is compact and uncountable, then \( \wadge{X} \) is isomorphic to \( \wadge{\cantor} \).
\end{enumerate-(1)}
\end{proposition}

\begin{remark} 
Since \(\sigma\)-compactness already appeared in~\cite[Theorem 5.12(2)]{mottorosquasiPolish} as a dividing line for Wadge-like hierarchies on zero-dimensional Polish space, one may be tempted to conjecture that \cref{prop:isomorphicWadgethroughretractions}\ref{prop:isomorphicWadgethroughretractions-1} can be turned into a characterization. We will however show that this is not the case: by \cref{thm:maintheorem}, there are \(\sigma\)-compact spaces whose Wadge hierarchy is isomorphic to \( \wadge{\baire} \) (take any \(\sigma\)-compact space with non-compact perfect kernel, e.g.\ \( \omega \times \cantor \)), and there are non-compact spaces whose Wadge hierarchy is isomorphic to \( \wadge{\cantor} \) (take any non-compact space with compact kernel and finite CB-rank, e.g.\ \( \omega \oplus \cantor \) where \( \oplus \) denotes the disjoint sum and \( \omega \) is discrete). 
\end{remark}

\subsection{More on the relativization method}

The next lemma shows that for any retraction \( \rho \colon \baire \to X \), the map \( \nsdclass \mapsto \rho^{-1}(\nsdclass) \) from \( \wc{X} \) to \( \wc{\baire} \) is a right inverse of (the restriction to \( \wc{\baire} \) of) the relativization map \( \nsdclass \mapsto \nsdclass(X) \). The proof is the same as the first part of the proof of~\cite[Lemma 6.5]{articolo_raphael} and does not require extra assumptions beyond \( \ZF+ \DC(\RR) \).

\begin{lemma} \label{lem:relativization} 
Let \( X \subseteq \baire \) be closed, and \( \rho \colon \baire \to X \) be a retraction. Then for every \( \mathbf{\Lambda} \in \wc{X} \)
\[ 
(\rho^{-1}(\mathbf{\Lambda}))(X) = \mathbf{\Lambda}.
 \] 
\end{lemma}

\begin{proof}
Let \( A \subseteq X \) be complete for \( \mathbf{\Lambda} \). If \( B \in (\rho^{-1}(\mathbf{\Lambda}))(X) \) then  \( \rho^{-1}(B) \in \rho^{-1}(\mathbf{\Lambda}) \) because  \( \rho \) is continuous, thus \( \rho^{-1}(B) \lew^ \baire \rho^{-1}(A) \) by definition of \( \rho^{-1}(\mathbf{\Lambda}) \), hence \( B \lew^X A \), and finally \( B \in \mathbf{\Lambda} \) by choice of \( A \).

Conversely, pick any \( B \in \mathbf{\Lambda} \),  let \( f \colon X \to X \) witness \( B \lew^X A \), and let \( g \colon \baire \to X \) be any continuous function. Then \( f \circ g \) witnesses \( g^{-1}(B) \lew^{\baire,X} A \), and since \( A \eqw ^{X,\baire} \rho^{-1}(A) \) we get \( g^{-1}(B) \lew^{\baire} \rho^{-1}(A) \), i.e.\ \( g^{-1}(B) \in \rho^{-1}(\mathbf{\Lambda}) \). Since \( g \) was arbitrary, this shows that \( B \in (\rho^{-1}(\mathbf{\Lambda}))(X) \), as desired.
\end{proof}

\begin{corollary}[\( \AD \)] \label{cor:relativization} 
Let \( X \subseteq Y \) be zero-dimensional Polish spaces with \( X \) closed in \( Y \), and let \( \nsdclass \subseteq \pow(\baire) \) be a boldface pointclass. If \( \nsdclass(X) \in \nsdc{X} \), then \( \nsdclass(Y) \in \nsdc{Y} \). Moreover, if \( A \subseteq X \) is complete for \( \nsdclass(X) \) and \( \rho \colon Y \to X \) is any retraction, then \( \rho^{-1}(A) \) is complete for \( \nsdclass(Y) \).
\end{corollary}

\begin{proof}
Without loss of generality, we may assume that \( Y \) (and hence also \( X \)) is a closed subset of \( \baire \). Let \( A \) and \( \rho \) be as in the statement, and fix a retraction \( \rho' \colon \baire \to Y \), so that \( \rho \circ \rho' \) is a retraction of \( \baire \) onto \( X \). Then \( (\rho \circ \rho')^{-1}(\nsdclass(X)) = \nsdclass \) by \cref{lem:relativization} applied to \( \mathbf{\Lambda} = \nsdclass(X) \) and \cref{relativization}\ref{relativization-2}. Using this and applying once again \cref{lem:relativization} with  
\( \mathbf{\Lambda} = \rho^{-1}(\nsdclass(X)) \in \wc{Y} \), it follows that 
\[ 
\rho^{-1}(\nsdclass(X)) = ((\rho')^{-1}(\rho^{-1}(\nsdclass(X))))(Y) = ((\rho \circ \rho')^{-1}(\nsdclass(X))) (Y) =   \nsdclass(Y) , 
\]
which proves at once that \( \nsdclass(Y) \) is nonselfdual (because so is \( \rho^{-1}(\nsdclass(X)) \)) and that \( \rho^{-1}(A) \) is complete for \( \nsdclass(Y) \).
\end{proof}

The relativization method and, in particular, \cref{cor:relativization}, gives a general version of the classical computation in $\baire$
of the supremum of at most countably many nonselfdual classes, that we use several times in this paper.

\begin{lemma}[\( \AD \)] \label{calculus_sup}
Let $X$ be a zero-dimensional Polish space, $(X_i)_{i\in I}$ be a clopen partition of $X$  for some $I\subseteq\omega$,
and $(\nsdclass_i)_{i\in I}$ be boldface pointclasses in $\baire$.
Suppose that  $\nsdclass_i(X_i) \in \nsdc{X_i}$ for all \( i \in I \),  and let $A_i\subseteq X_i$ be complete for $\nsdclass_i(X_i)$.
Then the pointclass $\nsdclass=\wadgeclass{A}_X \in \wc{X}$ generated by  $A=\bigcup_{i \in I} A_i$ is the smallest Wadge class such that 
 $\nsdclass \supseteq \nsdclass_i(X)$ for all \( i \in I \).

Moreover, if for every \( i \in I \) there is \( j \in I \) such that \( \widecheck{\nsdclass}_i \subseteq \nsdclass_j \), then \( \nsdclass \in \sdc{X} \).
\end{lemma}

The class \( \nsdclass=\wadgeclass{A}_X \) from the lemma is called the \emph{supremum} of the  $\nsdclass_i$'s in $X$.

\begin{proof}
For each \( i \in I \) pick a retraction \( \rho_i \colon X \to X_i \), so that
\( \nsdclass_i(X) = \wadgeclass{\rho_i^{-1}(A_i)}_X \) by \cref{cor:relativization}. Since the identity function witnesses \( A_i \lew^{X_i,X} A \) and \(\rho_i^{-1}(A_i) \eqw^{X,X_i} A_i \), we get \( \nsdclass_i(X) \subseteq \wadgeclass{A}_X =  \nsdclass \). Moreover, if \( B \subseteq X \) is such that \( \wadgeclass{B}_X \supseteq \nsdclass_i(X) \), then \( \rho_i^{-1}(A_i) \lew^X B \), and hence \( A_i \lew^{X_i,X} B \) via some \( f_i \colon X_i \to X \). It follows that \( \bigcup_{i \in I } f_i \) witnesses \( A \lew^X B \), and thus \( \nsdclass \subseteq \wadgeclass{B}_X \).

For the second part, notice that we can assume without loss of generality that \( X \) is closed in \( \baire \), so that the \( X_i \)'s are closed in \( \baire \) as well. Let \( r_i \colon \baire \to X_i \) be retractions. If \( i,j \in I \) are as in the statement, then using \cref{cor:relativization} again we get
\[ 
X_i \setminus A_i \eqw^{X_i,\baire} {\baire \setminus r_i^{-1}(A_i)} \lew^{\baire} r_{j}^{-1}(A_{j}) \eqw^{\baire,X_{j}} A_{j},
 \] 
where the middle reduction exists because \( \baire \setminus r_i^{-1}(A_i) \in \widecheck{\nsdclass}_i \subseteq \nsdclass_{j} \). If \( f_i \colon X_i \to X_{j} \) witnesses \(  X_i \setminus A_i \lew^{X_i,X_{j}} A_{j} \), then \( \bigcup_{i \in I} f_i \) witnesses \( \neg A \lew^X  A \), as desired.
\end{proof}

The typical situations in which we will apply the ``moreover'' part of \cref{calculus_sup} are when 
$I = 2$ and $\nsdclass_0 = \widecheck{\nsdclass}_1$, or $I=\omega$ and $\nsdclass_i \subsetneq \nsdclass_{i+1}$ (in the second case use Wadge's Lemma to obtain \( \widecheck{\nsdclass}_i \subseteq \nsdclass_{i+1} \)). Notice also that if $X$ is uncountable, then by \cref{relativization}\ref{relativization-3} the hypothesis $\nsdclass_i(X) \in \nsdc{X_i}$ follows from the more manageable \( \nsdclass_i \in \nsdc{\baire} \).

\subsection{Compact rank} \label{subsec:compactrank}

An important dividing line in our analysis is whether the given Polish space \(X \) has a compact perfect kernel or not.
If \( X \) has a compact perfect kernel, then its iterated Cantor-Bendixson derivatives are eventually compact.
We use this idea to define a new ordinal invariant, called compact rank.

\begin{defin}
Let \( X \) be a Polish space such that \( \kerp{X} \) is compact. Then the \emph{compact rank} of \( X \) is 
\[
\rkcomp{X} = \min \{\alpha \leq \rank{X}_{\cb} \mid \derrkcb{\alpha}{X} \text{ is compact}\}.
\]
\end{defin}

The assumption on \( \kerp{X} \) ensures that \( \rkcomp{X} \) is well-defined. Moreover, \( \rkcomp{X} = 0 \) if and only if the whole \( X \)  is compact.

There is a natural characterization of when a space with compact perfect kernel is not compact itself which will be useful later on. It
 is based on the following well-known exercise.%
 \footnote{The nontrivial direction follows from the fact that if \( \{ U_n \mid n \in \omega \} \) is a countable clopen cover of the space which does not have a finite subcover, then setting \( V_n = U_n \setminus \bigcup_{j < n} U_j \) we get that \( \{ V_n \mid n \in \omega \wedge V_n \neq \emptyset \} \) is an infinite clopen partition of the space.}

\begin{fact}\label{partition_clopen_zero_dimensional} 
A zero-dimensional second-countable space is not compact if and only if it admits an infinite (countable) clopen partition.
\end{fact}

\begin{lemma}\label{infinite_discrete_clopen_in_nocompact}
Suppose that $X$ is a  zero-dimensional Polish space with $\kerp{X}$ compact.
Then $X$ is not compact if and only if there exists an infinite discrete clopen set  $U \subseteq X$.
\end{lemma}

\begin{proof}
	If $U$ is infinite discrete then it is not compact, and if it is closed in \( X \) then \( X \) cannot be compact either.
	
	Conversely, suppose that $X$ is not compact. By \cref{partition_clopen_zero_dimensional}, there exists an infinite clopen partition $(C_n)_{n \in \omega}$ of $X$.
	Since $\kerp{Z}$ is compact, there exists $m \in \omega$ such that $\kerp{Z}\subseteq\bigcup_{n\leq m}C_n$.
	For each $n>m$ there is at least one isolated point \( x_n \in C_n \) because \( C_n \cap \kerp{X} = \emptyset \): then $U = \{x_n \mid n > m\}$ is as required.
\end{proof}

It might be interesting to notice that the compact rank can be obtained through a corresponding derivative \( D_{\comp}(F) \) defined by 
\begin{equation*}
	D_{\comp}(F) = F \setminus \bigcup\{U \mid U\text{ is infinite  discrete clopen in }F\}.
\end{equation*}
(With this terminology, \cref{infinite_discrete_clopen_in_nocompact} can be reformulated as: If \( \kerp{X} \) is compact, then $X$ is compact  if and only if $D_{\comp}(X)=X$.) One can then prove that if $X$ is a non-compact zero-dimensional Polish space with $\kerp{X}$ compact, then $D_{\cb}(X) = D_{\comp}(X)$, and thus \( \rkcomp{X} \) is the rank associated to the derivative \( D_{\comp} \), i.e.\
\(
\rkcomp{X} = \min \{\alpha < \omega_1 \mid \dercomp{\alpha}{X}  = \dercomp{\alpha+1}{X} \}.
\)

We will also need the following technical lemma.

\begin{lemma}\label{sequence_clopen_derivativeto_alpha} 
Suppose that $X$ is a non-compact zero-dimensional Polish space with $\kerp{X}$ compact. 	
	Then for all $\alpha < \rkcomp{X}$, there exists an infinite clopen partition $(C_n)_{n \in \omega}$ of $X$ such that
	$C_0 \supseteq \derrkcb{\alpha+1}{X}$ and $\rkcb{C_{n+1}}=\alpha+1$ for all $n\in\omega$. 
\end{lemma}

\begin{proof}
Since $\alpha < \rkcomp{X} \leq \rkcb{X}$ we have
	$\derrkcb{\alpha+1}{X} \subsetneq  \derrkcb{\alpha}{X}$ and the latter is not compact. By \cref{infinite_discrete_clopen_in_nocompact}  there is an infinite discrete clopen (relatively to \( \derrkcb{\alpha}{X} \)) set
	$U = \{ x_i \mid i \in \omega \} \subseteq \derrkcb{\alpha}{X}$, so that \( U \cap \derrkcb{\alpha+1}{X} = \emptyset \). 
Let \( (U_n)_{n \in \omega} \) be open sets of \( X \) such that \( U_0  \cap \derrkcb{\alpha}{X} = 	\derrkcb{\alpha}{X} \setminus U \) and \( U_{i+1} \cap \derrkcb{\alpha}{X} = \{ x_i \} \). Without loss of generality, we may assume that \( \bigcup_{n \in \omega} U_n = X \) (otherwise we replace \( U_0 \) with \( U_0 \cup (X \setminus \derrkcb{\alpha}{X} \)). Using the generalized reduction property (\cite[Theorem 22.16]{kechris}) we obtain a clopen partition \( (C_n)_{n \in \omega} \) of \( X \) such that \( C_n \subseteq U_n \), so that in particular \( C_n \cap \derrkcb{\alpha}{X} = U_n \cap \derrkcb{\alpha}{X} \)
because \( (U_n \cap \derrkcb{\alpha}{X})_{n \in \omega} \) is a clopen partition of \( \derrkcb{\alpha}{X} \). Since \( U \cap \derrkcb{\alpha+1}{X} = \emptyset \), by \cref{clopen_rkgb_no_change} the clopen partition $(C_n)_{n \in \omega}$ of \( X \) is as desired.
\end{proof}

\subsection{Minimal countable spaces}

We define (up to homeomorphism) a canonical sequence \( (K_{\alpha+1})_{\alpha < \omega_1} \) of countable compact metrizable spaces such that each \( K_{\alpha+1} \) has CB-type \( (\alpha+1,1) \) (thus it is simple) and embeds into every Polish space of CB-rank at least \( \alpha +1 \).

\begin{defin}\label{definitionKalpha}
	Set $K_0=\emptyset$ and inductively define the space $K_{\alpha+1}$ for all $\alpha< \omega_1$ as follows:
	\begin{itemizenew}
	\item 
	$K_{\alpha+1} = \ptglsq{K_{\alpha}}$, when $\alpha = 0$ or \(\alpha\) is a successor ordinal;
	 \item 
	 $K_{\alpha+1} = \ptglsq{(K_{\alpha_n+1})_{n \in \omega}}$ for some $(\alpha_n)_{n \in \omega}$ cofinal in $\alpha$, when \(\alpha\) is limit.
	\end{itemizenew}
\end{defin} 
Note that $K_1=\set{\inftycostbaire{0}}$. By induction on $\alpha<\omega_1$,
one easily sees that \( K_{\alpha+1} \) is compact and $\cbtype(K_{\alpha+1})=(\alpha+1,1)$.
As a consequence, since the CB-type is a complete invariant for compact countable Polish spaces the definition of $K_{\alpha+1}$
when $\alpha$ is limit does not depend on the choice of the cofinal sequence $(\alpha_n)_{n\in\omega}$, so the sequence is well-defined up to homeomorphism.%
\footnote{Recall that in models of \( \AD \), it is not possible to define an \( \omega_1 \)-sequence of distinct  compact sets because, by the Perfect Set Property (see e.g. \cite[Theorem 33.3]{JechSetTheory}), no uncountable subset of a standard Borel space can be well-orderable: this is why we insist here that the definition is given only up to homeomorphism. What we will use is just the existence, for any fixed \(\alpha < \omega_1 \), of a countable compact space \( K_{\alpha+1} \) with CB-type \( (\alpha+1,1) \) and the way it can be constructed from analogous spaces of lower CB-rank, but its actual presentation as a subspace of \( \baire \) will be totally irrelevant.}

\begin{proposition}\label{Kalpha_embed_countable}
	For every ordinal $\alpha < \omega_1$ and every Polish space $X$ with $\rkcb{X} \geq \alpha + 1$
	there exists an embedding $i \colon K_{\alpha + 1} \to X \setminus \kerp{X}$ with compact (hence closed) range.
\end{proposition}

\begin{proof}
Since \( K_{\alpha+1} \) is compact, it is enough to prove by  induction on \(\alpha\) that there is a continuous injection $i \colon K_{\alpha + 1} \to X \setminus \kerp{X}$. 

If \( \alpha = 0 \) then \( K_1 = \{ \inftycostbaire{0} \} \) and \( X \) has at least one isolated point \( x \): setting \( i(\inftycostbaire{0}) = x \) we are done.

Assume now that \( \alpha > 0 \). Since \( \rkcb{X} \geq \alpha+1 \), then \( \derrkcb{\alpha}{X} \setminus \derrkcb{\alpha+1}{X} \neq \emptyset \). Pick any \( x \in \derrkcb{\alpha}{X} \setminus \derrkcb{\alpha+1}{X} \) 	and fix any \( \subseteq \)-decreasing clopen basis $(V_n)_{n\in\omega}$ of neighborhoods of $x$ with 
	\( V_0 \cap \derrkcb{\alpha}{X} = \{ x \} \) (hence also \( V_0 \cap \derrkcb{\alpha+1}{X} = \emptyset \)). Set \( U_n = V_n \setminus V_{n+1} \), so that each \( U_n \) is a countable clopen set, \( U_n \to x \), \( \rkcb{U_n} \leq \alpha \) because \( U_n \cap \derrkcb{\alpha}{X} = \emptyset \),  and \( \limsup_{n \in \omega} \rkcb{U_n} = \alpha \)%
\footnote{That is, \( \forall \beta < \alpha \forall n\in \omega \exists m \geq n \, (\rkcb{U_n} > \beta )\).}
 because every clopen \( x \in V \subseteq V_0 \) is such that \( \rkcb{V} = \alpha+1 \). (Here we repeatedly use \cref{clopen_rkgb_no_change}.) 
Let \( (\alpha_n)_{n \in \omega} \) be the constant sequence with value \( \alpha \) if the latter is a successor ordinal, and a sequence of successor ordinals cofinal in \(\alpha\) otherwise. By \( \limsup_{n \in \omega} \rkcb{U_n} = \alpha \),  there is an injective \( f \colon \omega \to \omega \) such that \( \rkcb{U_{f(n)}} \geq \alpha_n \) for all \( n \in \omega \), so that by inductive hypothesis there is a continuous injection \( i_n \colon K_{\alpha_n} \to U_{f(n)} \).
Then the map \( i \colon K_{\alpha+1} \to V_0 \subseteq X \setminus\kerp{X} \) defined by
\[ 
i(y) = 
\begin{cases}
x & \text{if } y = \inftycostbaire{0} \\
i_n(z) & \text{if } y = (0)^n {}^\smallfrown{} (1) {}^\smallfrown{} z \text{ with } z \in K_{\alpha_n}
\end{cases}
 \] 
is the desired continuous injection.
\end{proof}

\subsection{More on the difference hierarchy}

Section~\ref{subsec:diffhierarchy} and
\cref{relativization} yields  a description of the first \( \omega_1 \)-many nonselfdual Wadge classes in \( \wadge{X} \) when \( X \) is uncountable. (Here we are not requiring \( \AD \) because Borel determinacy suffices.)

\begin{lemma} \label{lem:diffclassesinuncountable}
Let \( X \) be an uncountable zero-dimensional Polish space and \( \alpha < \omega_1 \). Then the \( \alpha \)-th nonselfdual ordinal in \( X \) is precisely the Wadge rank of  \( \differenceopenX{\alpha}{X} \).
\end{lemma}

Moving to countable Polish spaces \( X \), we instead observe that the difference hierarchy trivializes at least from \( \rkcb{X} \) on. Recall that we only consider nonempty spaces \( X \).

\begin{proposition}\label{countable_space_each_set_is_difference}
	Let $X$ be a  
	countable Polish space. If $\rkcb{X} \le \alpha$ then $\pow(X) = \differenceopenX{\alpha}{X}$, and thus \( \pow(X) = \differenceopenX{\alpha}{X} \cap \differenceopenXcheck{\alpha}{X} \).
		Moreover, if $X$ is simple and $\rkcb{X}=\beta+1$, then $\pow(X) = \differenceopenX{\beta}{X}\cup\differenceopenXcheck{\beta}{X}$.
\end{proposition}

\begin{proof}
Since \( \differenceopenX{\gamma}{X} \subseteq \differenceopenX{\gamma'}{X} \) if \( \gamma \leq \gamma' \), it is enough  to consider the case \( \rkcb{X} = \alpha \) and prove by induction that $\pow(X) = \differenceopenX{\alpha}{X}$.
	Note that \( 1 \leq \alpha < \omega_1 \) and $\derrkcb{\alpha}{X}=\emptyset$ because \( X \neq \emptyset \) is countable.

	If $\alpha = 1$ then $X$ is discrete, so  $\pow(X) = \boldsymbol{\Sigma}^0_1(X) =  \differenceopenX{1}{X}$.
	If moreover $X$ is simple, then \( X \) is a singleton and $\pow(X) = \set{\emptyset,X}= \differenceopenX{0}{X}\cup\differenceopenXcheck{0}{X}$.
	
	Suppose now that $\alpha = \beta + 1$ for some $1 \leq \beta < \omega_1$, and consider an arbitrary  $A\subseteq X$.
	Let  $Y =X \setminus \derrkcb{\beta}{X}$ and $A'=Y \setminus A$.
	Since $\rkcb{Y} = \beta$ by~\cref{clopen_rkgb_no_change}, we get $A' \in \differenceopenX{\beta}{Y}\subseteq \differenceopenX{\beta}{X}$ by inductive hypothesis and the fact that
 $Y$ is open in $X$, so we can write \( A' = \differenceseq{\beta}{(U_\gamma)_{\gamma < \beta}}\) with \( U_\gamma \in \boldsymbol{\Sigma}^0_1(X) \). Moreover,  
	$\derrkcb{\beta}{X}$ is  discrete  because $\rkcb{X} = \beta + 1$, therefore $A \cap \derrkcb{\beta}{X}$ is open in $\derrkcb{\beta}{X}$: 
let $U \in \boldsymbol{\Sigma}^0_1(X)$ be such that  $U \cap \derrkcb{\beta}{X} = A \cap \derrkcb{\beta}{X}$, so that the open subset \( U_\beta = U \cup Y \) of \( X \) satisfies $A = U_\beta \setminus A'$: then \( A = \differenceseq{\beta+1}{(U_\gamma)_{\gamma < \beta+1}} \), 
witnessing
	 $A \in \differenceopenX{\alpha}{X}$.
	If moreover $X$ is simple, then $\derrkcb{\beta}{X}=\set{x}$ is a singleton. 
	If \( x \notin A \), then \( A \subseteq Y \) and using the induction hypothesis we again obtain \( A \in \differenceopenX{\beta}{Y}\subseteq \differenceopenX{\beta}{X} \). If instead \( x \in A \), then \( A = X \setminus A' \), hence \( A \in \differenceopenXcheck{\beta}{X}\). In both cases, \( A \in 
	\differenceopenX{\beta}{X}\cup\differenceopenXcheck{\beta}{X} \).
		
	Finally, suppose that $\alpha$ is limit, and take a strictly increasing sequence $(\alpha_n)_{n \in \omega}$ of nonzero ordinals cofinal in $\alpha$.
	For $n\in\omega$, consider the open sets $V_n^* = X \setminus \derrkcb{\alpha_n}{X}$.
	Apply the generalized reduction property to the sequence $(V_n^*)_{n \in \omega}$ to
	obtain a clopen partition $(V_n)_{n \in \omega}$ of \( X \) such that $V_n \subseteq V_n^*$.
	Given any $A\subseteq X$, set $A_n = A \cap V_n$. Then $A_n \subseteq V_n$ and $\rkcb{V_n} \le \alpha_n$ by~\cref{clopen_rkgb_no_change},
	so by induction hypothesis $A_n \in \differenceopenX{\alpha_n}{V_n} \subseteq \differenceopenX{\alpha_n}{X}\subseteq \differenceopenX{\alpha}{X}$. Write each \( A_n \) as \( A_n = \differenceseq{\alpha}{(U^n_\beta)_{\beta<\alpha}} \) with \( U_\beta^n \in \boldsymbol{\Sigma}^0_1(X) \). Then  $ A = \bigcup_{n \in \omega} (A_n \cap V_n) = \differenceseq{\alpha}{(\bigcup_{n \in \omega} (U^n_\beta \cap V_n))_{\beta < \alpha}}$, hence 
	$A \in \differenceopenX{\alpha}{X}$. 
\end{proof}

Notice that if \( \nsdclass \in \nsdc{X} \) and \( \nsdclass \subseteq \boldsymbol{\Delta}^0_2(X) \) for a given zero-dimensional Polish space \( X \), then	 
 $\nsdclass = \differenceopenX{\beta}{X}$ or $\nsdclass = \differenceopenXcheck{\beta}{X}$ for some \( \beta < \omega_1 \). (Recall that by convention $\differenceopenX{0}{X}=\set{\emptyset}$.)
Indeed, if \( \rho \colon \baire \to X \) is any retraction then \( \rho^{-1}(\nsdclass) \subseteq \boldsymbol{\Delta}^0_2(\baire) \) is nonselfdual, and thus  $\rho^{-1}(\nsdclass) = \differenceopenX{\beta}{\baire}$ or $\rho^{-1}(\nsdclass) = \differenceopenXcheck{\beta}{\baire}$ for some \( \beta < \omega_1 \) by  Wadge's analysis of the first \( \omega_1 \)-many nonselfdual pairs in \( \wadge{\baire} \): since \( \nsdclass = (\rho^{-1}(\nsdclass))(X) \) by \cref{lem:relativization}, we are done. This reproves  \cite[Theorem 11.2]{articolo_raphael}, and since \( \pow(X) \) is selfdual then  \cref{countable_space_each_set_is_difference} sharpens it by providing an upper bound on the possible $\beta$'s when \( X \) is countable.

\begin{corollary}\label{identification_delta02_nsdclass_differences}
	Let $X$ be a countable Polish space and $\nsdclass \in \nsdc{X}$. Then $\nsdclass = \differenceopenX{\beta}{X}$ or $\nsdclass = \differenceopenXcheck{\beta}{X}$ \emph{for some $\beta < \rkcb{X}$}.
\end{corollary}

We are now going to show that such bound is optimal.

\begin{defin}
	Given $\alpha<\omega_1$,  set 
	\[ 
A_\alpha  =\set{x\in K_{\alpha+1}}[\rkcbx{x}{K_{\alpha+1}}\mbox{ has  parity opposite to }\alpha].
 \] 
 \end{defin}
	
\begin{proposition}\label{Difference_hierarchy_Kalpha}
For all $\alpha<\omega_1$, $A_\alpha\in \differenceopenX{\alpha}{K_{\alpha+1}}\setminus\differenceopenXcheck{\alpha}{K_{\alpha+1}}$.
\end{proposition}

\begin{proof}
By definition, \( A_0 = \emptyset \) and \( K_{\alpha+1} \setminus A_\alpha = \ptglsq{ A_\beta} \) if \( \alpha = \beta+1 \). Moreover, since when \(\alpha\) is limit the choice of the sequence \( (\alpha_n)_{n \in \omega} \) in the definition of \( K_{\alpha+1} = \ptglsq{(K_{\alpha_n+1})_{n\in \omega}} \) is irrelevant (up to homeomorphism), we can assume that the ordinals \( \alpha_n \) are all odd, so that also in the limit case
\( K_{\alpha+1} \setminus A_\alpha = \ptglsq{(A_{\alpha_n})_{n \in \omega}} \). 

Recall the Wadge degrees \( \mathbf{D}_\alpha \in \wadge{\baire} \) from~\eqref{eq:differencedegrees}. Since \( \differenceopen{\alpha} \in \nsdc{\baire} \), it is enough to prove by induction  that \( A_\alpha \eqw^{K_{\alpha+1}, \baire} B_\alpha \) for some/any \( B_\alpha \in \mathbf{D}_\alpha \). If \( \alpha = 0 \) the result is obvious because \( A_0 = \emptyset = B_0 \). Suppose now that \( \alpha = \beta+1 \) and fix \( B_\beta \in \mathbf{D}_\beta \). Let
 \( f \colon K_{\beta+1} \to \baire \) and \( g \colon \baire \to K_{\beta+1} \) witness \( A_\beta \lew^{K_{\beta+1},\baire} B_\beta \) and \( B_\beta \lew^{\baire,K_{\beta+1}} A_\beta \), respectively. Let \( B_\alpha \in \mathbf{D}_\alpha \) be defined as in~\eqref{eq:differencedegrees}, that is \( B_\alpha = \{ (0)^n \conc (m+1) \conc x \mid n,m \in \omega \wedge x \in \neg B_\beta \} \). Using \( K_{\alpha+1} \setminus A_\alpha = \ptglsq{ A_\beta} \) it is straightforward to see that \( \ptglsq{f} \colon K_{\alpha+1}	 \to \ptglsq{\baire} \subseteq \baire \) witnesses \( A_\alpha \lew^{K_{\alpha+1},\baire} B_\alpha \), and that the map \( \ptglsq{g} \circ \rho \colon \baire \to K_{\alpha+1} \), where \( \rho \colon \baire \to \ptglsq{\baire} \) is the retraction defined by \( \rho(\inftycostbaire{0}) = \inftycostbaire{0} \) and \( \rho( (0)^n \conc (m+1) \conc x) = (0)^n \conc (1) \conc x \), witnesses \( B_\alpha \lew^{\baire,K_{\alpha+1}} A_\alpha \).
 The limit case is similar.
\end{proof}

As a by-product, we obtain a computation of \(\rkcb{X} \) for countable spaces \( X \) in terms of difference classes.

\begin{proposition}\label{CB_and_differences}
	Let $X$ be a countable Polish space. Then
	\[
	\rkcb{X} = \min \{ \alpha<\omega_1 \mid \differenceopenX{\alpha}{X} = \differenceopenXcheck{\alpha}{X} \}.
	\]
\end{proposition}

\begin{proof}
By \cref{countable_space_each_set_is_difference} we know that \( \differenceopenX{\alpha}{X} = \pow(X) =  \differenceopenXcheck{\alpha}{X} \) for \( \alpha = \rkcb{X} \), so it is enough to show that \( \differenceopenX{\alpha}{X} \neq \differenceopenXcheck{\alpha}{X} \) for every  $\alpha< \rkcb{X}$.

	Since $\alpha+1\leq\rkcb{X}$, by \cref{Kalpha_embed_countable} we can suppose that $K_{\alpha+1}$ is, up to homeomorphism, a closed subset of $X$.
	By \cref{Difference_hierarchy_Kalpha} we get $A_{\alpha} \in \differenceopenX{\alpha}{K_{\alpha+1}} \setminus
	\differenceopenXcheck{\alpha}{K_{\alpha+1}}$.
	Fix a retraction \( \rho \colon X \to K_{\alpha+1} \) and set \( A'_\alpha = \rho^{-1}(A_\alpha) \): since \( A_\alpha \eqw^{K_{\alpha+1},X} A'_\alpha \), we get
	$A'_{\alpha} \in \differenceopenX{\alpha}{X} \setminus
	\differenceopenXcheck{\alpha}{X}$,
	as desired.
\end{proof}

Combining \cref{identification_delta02_nsdclass_differences} and \cref{CB_and_differences} we get a full description of nonselfdual Wadge classes in countable Polish spaces which nicely complements \cref{relativization}\ref{relativization-3} and \cref{lem:diffclassesinuncountable}.
Given any \( \alpha < \Theta \), we set \( \nsdc{\baire} \restriction \alpha = \{ \nsdclass \in \nsdc{\baire} \mid \rank{\nsdclass}_{\baire} < \alpha \} \).

\begin{corollary} \label{cor:diffclassesinuncountable}
Let \( X \) be a countable Polish space. Then the nonselfdual Wadge pointclasses in \( X \) are precisely \( \differenceopenX{\beta}{X} \) and their duals, for \( \beta < \rkcb{X} \).

In particular, \( \nsdc{X} = \{ \nsdclass(X) \mid \nsdclass \in \nsdc{\baire} \restriction 2 \cdot \rkcb{X} \} \), and
indeed \( \langle \nsdc{\baire} \restriction 2 \cdot \rkcb{X} , {\subseteq} \rangle \) and \( \langle \nsdc{X}, {\subseteq} \rangle \) are isomorphic via the map \( \nsdclass \mapsto  \nsdclass(X) \).
\end{corollary}

\section{Proof of \texorpdfstring{\cref{thm:maintheorem}}{Main Theorem 1}}

In view of the known facts from Propositions~\ref{prop:minimaldegree}, \ref{prop:semiwellordered}, and~\ref{prop:decompositionofslefdual},
to completely describe \( \wadge{X} \) (up to isomorphism) for a given zero-dimensional Polish space \( X \) we need to
\begin{enumerate-(1)}
\item \label{goal1}
determine which \( 3 \leq \alpha < \Theta_X \)  (if \( \Theta_X > 3 \)) are selfdual and which are not, namely, determine the sets \( \sdordinal{X} \) and \( \nsdordinal{X} \) from Definition~\ref{def:sdordinal};
\item \label{goal2}
determine the value of \( \Theta_X \) when \( X \) is countable.
\end{enumerate-(1)}

The first of these two goals splits into cases, depending on whether \( \alpha \) is successor or not. For successor ordinals  (Section~\ref{subsec:alternatingproperty}), we will show that, as in the case of \( \baire \) and \( \cantor \), nonselfdual pairs and selfdual degrees alternate. This is relatively easy for uncountable spaces, although the proof necessarily differs from the one used in the case of \( \baire \) and \( \cantor \) because we can no longer rely on the combinatorics of the space at hand, while in the countable case we need to perform a deeper analysis of the whole \( \wadge{X} \) and in particular of the maximal class(es).  

For \( \alpha \) limit (Section~\ref{subsec:limitlevels}), we will immediately observe that for all \( X \), if \( \alpha \) has uncountable cofinality then it is nonselfdual. 
The situation for limit ordinals of countable cofinality is  more delicate: they can be selfdual, as it happens in \( \baire \), or not, as in \( \cantor \).
It turns out that for a general zero-dimensional Polish space \( X \) the two behaviours (Baire-like and Cantor-like) can coexist, but in an orderly manner: if we denote by \( \countablecof \) the set of limit ordinals of countable cofinality, then \( \sdordinal{X} \cap \countablecof \) is always an initial segment of \( \Theta_X\cap \countablecof \), which might be empty, have length \( \Theta_X \) (i.e.\ \( \nsdordinal{X} \cap \countablecof = \emptyset \)), or
might be of the form \( \alpha_X \cap \countablecof \) for some  countable ordinal \( \alpha_X > \omega \).  

Finally, the deeper analysis of the (non)selfdual levels when \( X \) is countable will also allow us to compute the value of \( \Theta_X \) (Section~\ref{subsec:length}), thus completing the proof of \cref{thm:maintheorem}.

\subsection{The alternating property} \label{subsec:alternatingproperty}

We say that a zero-dimensional Polish space $X$ satisfies the \emph{alternating property} if for all \( \alpha \geq 1 \) such that $\alpha + 1 < \Theta_X$ we have
\[
\alpha \in \sdordinal{X} \iff \alpha + 1 \in \nsdordinal{X}.
\]
The goal of this subsection is to prove that all zero-dimensional Polish spaces satisfy the alternating property. The following proposition shows that there cannot be two consecutive selfdual Wadge classes. In the case of \( \baire \) and \( \cantor \) this is done constructively: there are explicit procedures to construct, starting from a given selfdual set, a nonselfdual pair immediately after it in the Wadge quasi-order. In the general case, instead, we cannot rely on any specific combinatorial property of the space at hand, thus we have to use a different and non-constructive proof.

\begin{proposition}[\( \AD \)]\label{non_selfdual_class_consecutive}
	Let $X$ be a zero-dimensional Polish space, and $A, A' \in \sds{X}$ be such that $A \slew^X A'$.
	Then there is $A'' \in \nsds{X}$ satisfying $A \slew^X A'' \slew^X A'$.
\end{proposition}

\begin{proof}
Apply \cref{prop:selfdualanalysis} to \( A' \) to get a clopen partition \( (V_n)_{n \in \omega} \) of \( X \) and nonselfdual sets \( (A_n)_{n \in \omega} \) such that \( A_n \slew^X A' \) and \( A' = \bigcup_{n \in \omega} (A_n \cap V_n ) \). If \( A_n \lew^X A \) for all \( n \in \omega \) via some continuous map \( f_n \colon X \to X \), then \( f = \bigcup_{n \in \omega} (f_n \restriction V_n) \) would witness \( A' \lew^X A \), a contradiction. Since \( A \) is selfdual, this means that by Wadge's Lemma there is \( \bar{n} \in \omega \) such that \( A \slew^X A_{\bar{n}} \). Setting \( A'' = A_{\bar{n}} \) we are done.
\end{proof}

In particular, \cref{non_selfdual_class_consecutive} already yields one implication of the alternating property.

\begin{corollary}[\( \AD \)]\label{after_sd_is_nsd}
Let $X$ be a zero-dimensional  Polish  space and $\alpha \geq 1$ be such that \( \alpha +1  < \Theta_X\). If $\alpha \in \sdordinal{X}$, then $\alpha + 1 \in \nsdordinal{X}$. 
\end{corollary}

The following corollary will be used in \cref{simple_is_NSD}. To simplify the notation let  \( \differencestaropenX{\beta}{X} = (\differenceopenX{\beta}{X})^* = \differenceopenX{\beta}{X} \cup \differenceopenXcheck{\beta}{X} \) be the coarse Wadge class associated to \( \differenceopenX{\beta}{X} \).

\begin{corollary} \label{cor:lastclass}
Let \( X \) be a countable%
\footnote{Since we are dealing with a countable space, Borel determinacy suffices to apply our previous results without appealing to the full \( \AD \).}
Polish space. Then all sets \( A \subseteq X \) such that \( A \notin \bigcup_{\beta < \rkcb{X}} \differencestaropenX{\beta}{X} \), if there exist any, are selfdual, Wadge equivalent to each other, and \( \pow(X) = \wadgeclass{A}_X \). 
\end{corollary}

\begin{proof}
By \cref{cor:diffclassesinuncountable}, all sets \( A \) as in the statement are selfdual. Assume towards a contradiction that there is \( A' \notin \bigcup_{\beta < \rkcb{X}} \differencestaropenX{\beta}{X} \) such that \( A' \not\lew^X A \), so that \( A' \in \sds{X} \) and \( A \slew^X A' \). If \( A' \) is \( \lew^X \)-minimal with such properties, then we get a contradiction with (the restriction to Borel sets of) \cref{non_selfdual_class_consecutive}. Thus \( A \eqw^X A' \) for all \( A,A' \notin \bigcup_{\beta < \rkcb{X}} \differencestaropenX{\beta}{X} \). 

Finally,
fix \( A \) as in the statement and
 consider an arbitrary \( B \subseteq X \). If \( B \in \differencestaropenX{\beta}{X} \) for some \( \beta < \rkcb{X}\), then \( B \slew^X A \) by Wadge's Lemma; otherwise, \( B \eqw^X A \) by the previous paragraph. In all cases \( B \lew^X A \), thus \( \pow(X) = \wadgeclass{A}_X \). 
\end{proof}

\begin{theorem}[\( \AD \)] \label{after_nsd_is_sd}
All uncountable  zero-dimensional Polish spaces satisfy the alternating property.
\end{theorem}

\begin{proof}
	Let $X$ be an uncountable Polish zero-dimensional space, and consider any $1 \leq \alpha < \Theta_X = \Theta$.
	By \cref{after_sd_is_nsd} we only need to show that if $\alpha \in \nsdordinal{X}$, then $\alpha + 1 \in \sdordinal{X}$. By \cref{relativization}\ref{relativization-2} there is \( \nsdclass \in \nsdc{\baire} \) such that \( \nsdclass(X) \in \nsdc{X} \) and \( \rank{\nsdclass(X)} = \alpha \). Since \( X \) is uncountable then \( \kerp{X} \) is uncountable as well, thus we can partition \( X \) into two uncountable clopen sets \( X_0, X_1 \). By \cref{relativization}\ref{relativization-3} we have \( \nsdclass(X_0) \in \nsdc{X_0} \) and \( \widecheck{\nsdclass}(X_1) \in \nsdc{X_1} \). By \cref{calculus_sup} the smallest Wadge class containing \( \nsdclass(X) \cup \widecheck{\nsdclass}(X) \) is selfdual, i.e.\ \( \alpha+1 \in \sdordinal{X} \).
\end{proof}

Combined with \cref{relativization} and \cref{lem:relativization}, the alternating property implies that the embedding from \( \wadge{X} \) into \( \wadge{\baire} \) induced by any retraction \( \rho \colon \baire \to X \) (Section~\ref{subsec:retraction}) is almost surjective and might miss only limit levels of \( \wadge{X} \).
In particular, except for the countable cofinality case the rank of \( \nsdclass \in \wc{\baire} \) and the rank (in \( X \)) of \( \nsdclass(X) \) are at distance at most $1$.

We now move to countable spaces.

\begin{proposition}\label{simple_is_NSD}
	Let $X$ be a countable Polish space. Then $\pow(X)$ is a Wadge class if and only if $X$ is not simple.
\end{proposition}

\begin{proof}
Suppose first that $X$ is simple, so that in particular $\rkcb{X}=\alpha+1$ for some \( \alpha < \omega_1 \). By \cref{cor:diffclassesinuncountable} we know that
	$\differenceopenX{\alpha}{X} \neq \differenceopenXcheck{\alpha}{X}$, while the ``moreover'' part of \cref{countable_space_each_set_is_difference} gives $\pow(X)=\differenceopenX{\alpha}{X} \cup \differenceopenXcheck{\alpha}{X}$. It follows that we cannot have \( \pow(X) = \wadgeclass{A}_X \) for \( A \subseteq X \), because if \( B \subseteq X \) is complete for \( \differenceopenX{\alpha}{X} \) then \( B, \neg B \lew^X A \) would imply \( B \slew^X A \) by \( B \not\lew^X \neg B \), a contradiction.
	
	Suppose now that $X$ is not simple.
By \cref{cor:lastclass} it is enough to show that there is some \( A \subseteq X \) such that 
\( A \notin \bigcup_{\beta<\rkcb{X}}\differencestaropenX{\beta}{X} \).
	We distinguish two cases. 
	
	If \( \rkcb{X} = \lambda \) is limit, fix a strictly increasing sequence \( (\alpha_i)_{i \in \omega} \) cofinal in \( \lambda \). By \cref{decomposition_in_simple_clopen} we can find a clopen partition \( (C_n)_{n \in \omega} \) of \( X \) such that \( \rkcb{C_n} < \lambda \) and \( \sup_{n \in \omega} \rkcb{C_n}= \lambda \). Fix an injection \( i \mapsto n_i \) such that \( \rkcb{C_{n_i}} > \alpha_i \), so that \( \differenceopenX{\alpha_i}{C_{n_i}} \in \nsdc{C_{n_i}} \) for all \( i \in \omega \) by \cref{cor:diffclassesinuncountable}. By \cref{calculus_sup} there is a selfdual Wadge class \( \wadgeclass{A}_X \) which contains \( \bigcup_{i \in \omega} \differenceopenX{\alpha_i}{X}\), hence \( A \notin \bigcup_{\beta < \lambda} \differencestaropenX{\beta}{X} \) because all the pointclasses \( \differenceopenX{\beta}{X} \) are distinct and nonselfdual by \cref{cor:diffclassesinuncountable}.
	
	Suppose now that \( \cbtype(X) = (\beta+1, N) \) with \( 1 < N \leq \omega \).
	By \cref{decomposition_in_simple_clopen} again we can find a clopen partition \( (C_n)_{n < N} \) of \( X \) such that \( \rkcb{C_n} = \beta+1 \), so that \( \nsdclass_0(C_0) = \differenceopenX{\beta}{C_0} \in \nsdc{C_0} \) and \( \nsdclass_1(C_1) = \differenceopenXcheck{\beta}{C_1} \in \nsdc{C_1} \) by \cref{cor:diffclassesinuncountable}. (For the remaining \( 1 < n < N \) the choice of the nonselfdual pointclasses \( \nsdclass_n(C_n)  \) is irrelevant.) Then \cref{calculus_sup} allows us to conclude that there is  a selfdual \( A \subseteq X \) with \( \wadgeclass{A}_X \supseteq \differenceopenX{\beta}{X} \cup \differenceopenXcheck{\beta}{X} \), which is again as required by \cref{cor:diffclassesinuncountable}.
\end{proof}

\begin{theorem}\label{alternating_duality_theorem_intro}
	All countable  zero-dimensional Polish spaces satisfy the alternating property.
\end{theorem}

\begin{proof}
Let $X$ be a countable Polish space.
    By \cref{after_sd_is_nsd}, we only need to prove that if $\alpha\in\nsdordinal{X}$
    and $\alpha+1<\Theta_X$ then $\alpha+1\in\sdordinal{X}$.
    By \cref{cor:diffclassesinuncountable} there exists a $\beta<\rkcb{X}$ such that
    $\rank{\differenceopenX{\beta}{X}}=\alpha$. We distinguish two cases.
    
Assume first that \( \beta + 1 < \rkcb{X} \).    
  Then $\derrkcb{\beta+1}{X} \neq \emptyset$ and $\derrkcb{\beta}{X}$ is infinite, which implies that we can split $X$ in two clopen sets
    $X_0$ and $X_1$ with $\rkcb{X_i} \geq \beta+1$ for $i=0,1$. (For example, pick \( x \in \derrkcb{\beta}{X} \setminus \derrkcb{\beta+1}{X} \) and \( X_0  \)  clopen such that \( X_0 \cap \derrkcb{\beta}{X}  = \{ x \} \): then \( X_0 \) and \( X_1 =X \setminus X_0 \) are as desired.)
    By \cref{cor:diffclassesinuncountable} we have that  $\nsdclass_0(X_0) = \differenceopenX{\beta}{X_0} \in \nsdc{X_0} \) and \( \nsdclass_1(X_1) =  \differenceopenXcheck{\beta}{X_1} \in \nsdc{X_1}$, so by  \cref{calculus_sup} 
    their supremum is selfdual and $\alpha+1\in\sdordinal{X}$.
    
Assume now \( \rkcb{X} = \beta+1 \). If \( X \) were simple, then \( \Theta_X = \rank{\differenceopenX{\beta}{X}} +1 = \alpha+1 \) by  \cref{countable_space_each_set_is_difference}, contradicting the choice of \(\alpha \). Thus \( X \) is not simple and \cref{simple_is_NSD} together with \cref{cor:lastclass} show that at level \( \alpha+ 1 \) we have the selfdual Wadge class \( \pow(X) \), so \( \alpha+1 \in \sdordinal{X} \) again and we are done.
\end{proof}

\subsection{Limit levels} \label{subsec:limitlevels}

We now move to the analysis of limit levels of \( \wadge{X} \). We first settle the uncountable cofinality case.

\begin{proposition}[\( \AD \)]\label{uncountablecof_ordinal_intro}
Let \( X \) be a zero-dimensional Polish space.
If $1 \leq \alpha < \Theta_X$ is a limit ordinal and $\alpha \in \sdordinal{X}$, then \( \cof{\alpha}= \omega \). Thus at limit levels of uncountable cofinality there is always a nonselfdual pair.
\end{proposition} 

\begin{proof}
Let \( A \in \sds{X} \) be such that \( \wrank{A}^X = \alpha \).	
By \cref{prop:selfdualanalysis},  there are sets \( (A_n)_{n \in \omega} \) such that \( \alpha = \sup_{n \in \omega} (\wrank{A_n}^X+1) \): this already shows that \( \alpha \) has countable cofinality.
\end{proof}

To deal with the countable cofinality case, we need to distinguish whether \( \kerp{X} \) is compact or not. We begin with a technical lemma, which might be of independent interest because it shows that \( \omega \times \cantor \) is minimal for closed embeddability among zero-dimensional Polish spaces with non-compact perfect kernel.

\begin{lemma} \label{prop:wadgeZ_eq_wadgeBaire_onedir}
Let $X$ be a zero-dimensional Polish space. The following are equivalent:
	\begin{enumerate-(i)}
	\item\label{wadgeZ_eq_wadgeBaire_onedir_ker_case} 
	$\kerp{X}$ is not compact;
		\item\label{wadgeZ_eq_wadgeBaire_onedir_partition_case} 
	there exists an infinite (countable) clopen partition \( (C_n)_{n \in \omega} \) of $X$ such that each \( C_n \) is uncountable;
	\item\label{wadgeZ_eq_wadgeBaire_onedir_embedding_case} 
\( X \) contains a closed set \( F \) homeomorphic to \( \omega \times \cantor \).
	\end{enumerate-(i)}
\end{lemma}

\begin{proof}
If \( X \) is countable, then all of \ref{wadgeZ_eq_wadgeBaire_onedir_ker_case}--\ref{wadgeZ_eq_wadgeBaire_onedir_embedding_case} are false, so without loss of generality we may assume that \( X \) is uncountable.

We start with \ref{wadgeZ_eq_wadgeBaire_onedir_ker_case} \( \Rightarrow \) \ref{wadgeZ_eq_wadgeBaire_onedir_partition_case}. Since \( \kerp{X} \) is not compact, by \cref{partition_clopen_zero_dimensional} there is an infinite clopen partition $(C'_n)_{n \in \omega}$ of $\kerp{Z}$, and since \( \kerp{X} \) is perfect each \( C'_n \) is uncountable. Thus it is enough to set \( C_n = \rho^{-1}(C'_n) \) where \( \rho \colon X \to \kerp{X} \) is any retraction.

Assume now that \ref{wadgeZ_eq_wadgeBaire_onedir_partition_case} holds. Since the \( C_n \)'s are uncountable Polish spaces, they contain a closed set \( F_n \subseteq C_n \) homeomorphic to \( \cantor \): then \( F = \bigcup_{n \in \omega} F_n \) is still closed and homeomorphic to \( \omega \times \cantor \), hence \ref{wadgeZ_eq_wadgeBaire_onedir_embedding_case} holds.

Finally, we prove \ref{wadgeZ_eq_wadgeBaire_onedir_embedding_case} \( \Rightarrow \) \ref{wadgeZ_eq_wadgeBaire_onedir_ker_case}. Since \( F \) is perfect, then \( F \subseteq \kerp{X} \). But then \( \kerp{X} \) cannot be compact, otherwise so would be \( F \) and \( \omega \times \cantor \), which is false.
\end{proof}

\begin{theorem}[\( \AD \)]\label{wadgeZ_eq_wadgeBaire_onedir}
	Let $X$ be an uncountable zero-dimensional Polish space, so that \( \Theta_X = \Theta \). Then 
	the following are equivalent:
	\begin{enumerate-(i)}
	\item	\label{wadgeZ_eq_wadgeBaire_onedir-i}
	\( \kerp{X} \) is not compact;
	\item \label{wadgeZ_eq_wadgeBaire_onedir-ii}
\( \alpha \in \sdordinal{X} \) for every limit \( \alpha < \omega_1 \);
	\item \label{wadgeZ_eq_wadgeBaire_onedir-iv}
	if \( \alpha \in \Theta_X \cap \countablecof \), then  \( \alpha \in \sdordinal{X} \).
	\end{enumerate-(i)}
\end{theorem}

\begin{proof}
It is enough to prove \ref{wadgeZ_eq_wadgeBaire_onedir-i} \( \Rightarrow \) \ref{wadgeZ_eq_wadgeBaire_onedir-iv} and (the contrapositive of) \ref{wadgeZ_eq_wadgeBaire_onedir-ii} \( \Rightarrow \) \ref{wadgeZ_eq_wadgeBaire_onedir-i}.

Assume first that \( \kerp{X} \) is not compact, so that by \cref{prop:wadgeZ_eq_wadgeBaire_onedir} there is an infinite clopen partition \( (C_n)_{n \in \omega} \) of \( X \) into uncountable pieces. Given any \( \alpha \in \Theta_X \in \countablecof \), pick a strictly increasing sequence of ordinals \( (\alpha_n)_{n \in \omega} \) cofinal in \(\alpha\) such that \( \alpha_n \in \nsdordinal{X} \) for all \( n \in \omega \) (which exists by the alternating property), and using \cref{relativization}\ref{relativization-2} let \( \nsdclass_n \in \nsdc{\baire} \) be such that \( \rank{\nsdclass_n(X)}_X = \alpha_n \). By \cref{relativization}\ref{relativization-3}  and uncountability of \( C_n \), we have \( \nsdclass_n(C_n) \in \nsdc{C_n} \) as well, hence by \cref{calculus_sup} the supremum \( \nsdclass \in \wc{X} \) of the pointlcasses \( \nsdclass_n(X) \) is selfdual, i.e.\ \( \alpha \in \sdordinal{X} \).

Assume now that \( \kerp{X} \) is compact and let \( \beta = \rkcb{X} \). Consider any \( \beta + \omega < \alpha < \omega_1 \): we want to show that if \( \alpha \in \sdordinal{X} \) then \( \alpha \) is a successor ordinal, so that any limit ordinal between \( \beta+\omega \) and \( \omega_1 \) witnesses the failure of \ref{wadgeZ_eq_wadgeBaire_onedir-ii}. Let \( A \subseteq X \) be a selfdual set such that \( \rank{A}_\w^X = \alpha \).
By \cref{prop:selfdualanalysis} there exists a clopen partition $(V_n)_{n \in \omega}$ of \( X \) and sets $(A_n)_{n \in \omega}$
	such that  $A_n \slew^X A = \bigcup_{n \in \omega} (A_n \cap V_n)$ and \( \alpha = \sup_{n \in \omega}(\alpha_n+1) \) for \( \alpha_n= \wrank{A_n}^X \). Recall also from \cref{rmk:sefldualanalysis}\ref{rmk:sefldualanalysis-1} that we can assume that for each \( n \in \omega \) either \( A_n = X \), or else \( A_n \subseteq V_n \). Since \( \kerp{X} \) is compact, there is \( N \in \omega \) such that \( V_n \cap \kerp{X} = \emptyset \) for all \( n \geq N \). Fix  any \( n \geq N \). Then \( \rkcb{V_n} \leq \beta \) by \cref{clopen_rkgb_no_change}, hence if \( A_n \subseteq V_n \) then \( A_n \in \differenceopenX{\beta}{V_n} \subseteq \differenceopenX{\beta}{X} \) by \cref{countable_space_each_set_is_difference}, which means that \( \alpha_n < \beta+\omega \) because by \cref{relativization} and the alternating property \( \rank{\differenceopenX{\beta}{X}}_X \leq \rank{\differenceopenX{\beta}{\baire}}_\baire = 2 \beta +1 \).
 If instead \( A_n = X \), then \( \alpha_n = 1 \). Hence \( \alpha_n + 1 < \beta+\omega \) for all \( n \geq N \). Since \(   \alpha > \beta+\omega \), this means that necessarily \( \alpha = \sup_{n \in \omega} (\alpha_n+1) = \sup_{n < N } (\alpha_n+1) = \max_{n < N} (\alpha_n+1) \), and thus \(\alpha\) is a successor ordinal.
\end{proof}

If \( X \) is countable, so that \( \kerp{X} = \emptyset \) is trivially compact, then \( \Theta_X < \omega_1 < \Theta = \Theta_{\baire} \) (see Section~\ref{subsec:length} for the exact computation of \( \Theta_X \)), and thus \( \wadge{X} \)  cannot be isomorphic to \( \wadge{\baire} \). Combining this with \cref{wadgeZ_eq_wadgeBaire_onedir} and our previous results (and recalling that, up to homeomorphism, any zero-dimensional Polish space can be construed as a closed subspace of \( \baire \)) we get: 

\begin{corollary}[\( \AD \)] \label{cor:wadgeZ_eq_wadgeBaire_onedir} 
If \( X \) is a zero-dimensional Polish space, then \( \wadge{X} \) is order-isomorphic to \( \wadge{\baire} \) if and only if \( \kerp{X} \) is not compact, where (when it exists) the isomorphism is canonically%
\footnote{Since we are working in models of \( \AD \), where the Axiom of Choice \( \AC \) fails, showing that two partial orders have the same description is not enough to conclude that they are actually isomorphic, so it is very important to observe that in our case the isomorphism can be realized in a very canonical and definable way.\label{canonicaliso}}
 induced by  \( A \mapsto \rho^{-1}(A) \) for some/any retraction \( \rho \colon \baire \to X \).
\end{corollary}

We now move to spaces \( X \) with a compact perfect kernel (this includes the case of countable spaces).
It turns out that the compact rank from Section~\ref{subsec:compactrank} provides a crucial dividing line.

\begin{theorem}[\( \AD \)]\label{countable_cof_and_rkcomp}
    Let $X$ be a zero-dimensional Polish space such that \( \kerp{X} \) is compact,
    and let $\alpha \in \Theta_X \cap \countablecof $.
	\begin{enumerate-(1)}
		\item \label{countable_cof_and_rkcomp-1}
		If $\alpha > \rkcomp{X}$, then $\alpha \in \nsdordinal{X}$.
		\item \label{countable_cof_and_rkcomp-2}
		If $\alpha < \rkcomp{X}$, then $\alpha \in \sdordinal{X}$.
		\item \label{countable_cof_and_rkcomp-3}
		If \( \alpha = \rkcomp{X} \), then \( \alpha \in \sdordinal{X} \) if and only if there is a clopen partition \( (V_n)_{n \in \omega} \) of \( X \) and \( N \in \omega \) such that \( \alpha = \sup_{n \geq N} (\rkcb{V_n}+1) \).
	\end{enumerate-(1)}
\end{theorem}

\begin{proof}
\ref{countable_cof_and_rkcomp-1}
Suppose towards a contradiction that \( \alpha > \rkcomp{X} \) but \( \alpha \in \sdordinal{X} \):
we adapt the proof of \cref{wadgeZ_eq_wadgeBaire_onedir} to show that \(\alpha\) 
needs to be a successor ordinal, contradicting \( \alpha \in \countablecof \).
Given any \( A \in \sds{X} \) such that \( \rank{A}_\w^X = \alpha \), use \cref{prop:selfdualanalysis} to get a clopen partition $(V_n)_{n \in \omega}$ of \( X \) and sets $(A_n)_{n \in \omega}$
	such that  $A_n \slew^X A = \bigcup_{n \in \omega} (A_n \cap V_n)$ and \( \alpha = \sup_{n \in \omega} (\alpha_n+1) \) for \( \alpha_n= \rank{A_n}_\w^X \).
	Since by definition $\derrkcb{\rkcomp{X}}{X}$ is compact, there exists $N \in \omega$ such that
	$V_n \subseteq X \setminus \derrkcb{\rkcomp{X}}{X}$ for all $n \geq N$. 
Using that either \( A_n = X \) or \( A_n \subseteq V_n \), this implies that \(  A_n \in \differenceopenX{\rkcomp{X}}{X} \) for all \( n \geq N \). Arguing as in \cref{wadgeZ_eq_wadgeBaire_onedir} (and using that \( \alpha > \rkcomp{X} \) is limit), we get that \( \rank{\differenceopenX{\rkcomp{X}}{X}}_X \leq \rank{\differenceopenX{\rkcomp{X}}{\baire}}_{\baire} = 2 \cdot \rkcomp{X}+1 < \alpha \), hence \( \sup_{n \geq N} (\alpha_n+1) < \alpha \). Thus \( \alpha = \sup_{n \in \omega}(\alpha_n+1) = \max_{n < N} (\alpha_n+1) \) is a successor ordinal, a contradiction. 

\ref{countable_cof_and_rkcomp-2}
    Suppose now that $\alpha<\rkcomp{X}$, so that in particular $X$ is not compact.
    Then by \cref{sequence_clopen_derivativeto_alpha} there is a clopen partition $(C_n)_n$ of $X$ such that
    $\derrkcb{\alpha+1}{X}\subseteq C_0$ and $\rkcb{C_{n+1}}=\alpha+1$ for all $n\in\omega$.
    For every \( n \geq 1 \) and \( \beta < \alpha \), \cref{cor:diffclassesinuncountable} 
    implies that \( \differencesigmaalphaX{\beta}{1}{C_n} \in \nsdc{C_n} \). Moreover, \(  \beta \leq \rank{\differencesigmaalphaX{\beta}{1}{X}} \leq 2\beta+1  \): 
    if \( X \) is countable this follows from the alternating property and \cref{cor:diffclassesinuncountable} together with \( \rkcb{X} \geq \rkcomp{X} > \alpha > \beta \), while if \( X \) is uncountable we can use the alternating property and \cref{lem:diffclassesinuncountable}. Since \(\alpha\) is limit, then \( 2\beta+1 < \alpha \) and hence \( \beta \leq \rank{\differencesigmaalphaX{\beta}{1}{X}}  < \alpha \). Therefore taking any strictly increasing sequence \( (\alpha_n)_{n \geq 1} \) cofinal in \( \alpha \) we have \( \differenceopenX{\alpha_n}{C_n} \in \nsdc{C_n} \), \( \rank{\differencesigmaalphaX{\alpha_n}{1}{X}}_{X} < \alpha \), and \( \sup_{n \geq 1} \rank{\differencesigmaalphaX{\alpha_n}{1}{X}}  = \alpha \). Setting \( \nsdclass_0 = \{ \baire \} \) and \( \nsdclass_n = \difference{\alpha_n}{\boldsymbol{\Sigma}^0_1} \) for \( n \geq 1 \), and applying \cref{calculus_sup}, we then get \( \alpha \in \sdordinal{X} \), as desired.
 
 \ref{countable_cof_and_rkcomp-3}
The proof of the desired equivalence is a variation of the arguments for parts~\ref{countable_cof_and_rkcomp-1} and~\ref{countable_cof_and_rkcomp-2}. Assume first that \( \alpha \in \sdordinal{X} \).
Given any \( A \in \sds{X} \) such that \( \wrank{A}^X=\alpha \), let \( (V_n)_{n \in \omega} \) and \( (A_n)_{n \in \omega} \) be as in \cref{prop:selfdualanalysis}. 
Since \( \derrkcb{\alpha}{X} \) is compact, there is \( N \in \omega \) such that \( V_n \cap  \derrkcb{\alpha}{X} = \emptyset \) for all \( n \geq N \), which also means 
 \( \rkcb{V_n} \leq \alpha \) by \cref{clopen_rkgb_no_change}.
 If \( \rkcb{V_n} = \alpha \) for some \( n \geq N \), then we apply \cref{decomposition_in_simple_clopen} to all such countable spaces  \( V_n \) and get the desired partition. (Recall that \(\alpha\) is limit.) So we can assume that \( \rkcb{V_n} + 1 \leq \alpha \) for all \( n \geq N \) and we only need to show that for every \( \beta < \alpha \) there is \( n \geq N \) with \( \rkcb{V_n} \geq \beta \). Assume not: since we can assume that either \( A_n = X \) or \( A_n \subseteq V_n \), by \cref{countable_space_each_set_is_difference} we would
have \( A_n \in \differenceopenX{\rkcb{V_n}}{V_n} \subseteq \differenceopenX{\beta}{V_n}  \subseteq \differenceopenX{\beta}{X} \) for all \( n \geq N \), and since the usual computation gives \( \rank{\differenceopenX{\beta}{X}}_X \leq 2 \beta+1 \) and \( 2 \beta +1 < \alpha \) because \(\alpha\) is limit, then \( \alpha = \sup_{n \in \omega} (\wrank{A_n}^X+1) = \max_{n < N} (\wrank{A_n}^X+1) \), contradicting the fact that \( \alpha \) is limit.

Conversely, let \( (V_n)_{n \in \omega} \) and \( N \in \omega \) be as in part~\ref{countable_cof_and_rkcomp-3}, and fix a sequence \( (\alpha_k)_{k \in \omega} \) cofinal in \( \alpha \). Since \(\alpha\) is limit, the assumption on \( (V_n)_{n \in \omega} \) implies that for each \( k \in \omega \) we can find \( n_k \geq N \) such that \( \alpha_k < \rkcb{V_{n_k}} \)  and further ensure that the map \( k \mapsto n_k \) is injective. By \cref{cor:diffclassesinuncountable} the class \( \nsdclass_{n_k}(V_{n_k})  = \differenceopenX{\alpha_k}{V_{n_k}} \) is nonselfdual in \( V_{n_k} \), and as in the proof of part~\ref{countable_cof_and_rkcomp-2} one can show that \( \alpha_{k} \leq \rank{\differenceopenX{\alpha_k}{X}} \leq 2 \alpha_k +1 < \alpha \). Setting \( \nsdclass_n = \{ \baire \} \) for all \( n \in \omega \) which are not of the form \( n_k \) and applying \cref{calculus_sup} we get \( \alpha \in \sdordinal{X} \), as desired.
\end{proof}

\begin{remark} \label{rmk:countable_cof_and_rkcomp}
Together with \cref{decomposition_in_simple_clopen}, \cref{countable_cof_and_rkcomp}\ref{countable_cof_and_rkcomp-3} shows that if \( X \) is a countable zero-dimensional Polish space and \( \rkcb{X} = \rkcomp{X} = \alpha < \Theta_X \) is limit, then necessarily \( \alpha \in \sdordinal{X} \).
\end{remark}

Concerning the characterization in \cref{countable_cof_and_rkcomp}\ref{countable_cof_and_rkcomp-3}, it is easy to provide examples of both sorts. 

\begin{example} \label{xmp:countable_cof_and_rkcomp}
Let \( \alpha \) be a countable limit ordinal and fix a strictly increasing sequence \( (\alpha_n)_{n \in \omega} \) cofinal in \( \alpha \). Denote by \( \bigoplus_{n \in \omega} X_n \) the (disjoint) sum of the topological spaces \( X_n \). (If \( X_n \subseteq \baire \), then up to homeomorphism the sum can be construed
 as \( \bigoplus_{n \in \omega} X_n= \bigcup_{n \in \omega} \{ n \conc x \mid n \in \omega \wedge x \in X_n \} \).) Then the spaces
\[ 
Y_\alpha = \bigoplus_{n \in \omega} K_{\alpha_n+1} \qquad \text{and} \qquad Z_\alpha = \ptglsq{(\omega \times K_{\alpha_n+1})_{n \in \omega}}
 \] 
are both countable zero-dimensional Polish spaces with compact rank \( \alpha \), and by
\cref{countable_cof_and_rkcomp}\ref{countable_cof_and_rkcomp-3} we have
  \( \alpha \in \sdordinal{Y_\alpha} \) but \( \alpha \in \nsdordinal{Z_\alpha} \). Moreover, \( \rkcb{Y_\alpha} = \alpha \) and \( \rkcb{Z_\alpha} = \alpha+1 \) (indeed, \( \cbtype(Z_\alpha) = (\alpha+1,1) \), i.e.\ \( Z_\alpha \) is also simple), and the values for such CB-ranks are as small as possible by \cref{rmk:countable_cof_and_rkcomp}.
  \end{example}

The following result, together with \cref{uncountablecof_ordinal_intro} and \cref{wadgeZ_eq_wadgeBaire_onedir}, completes the proof of part \ref{thm:maintheorem-4} of \cref{thm:maintheorem}. Let \( \alpha_X \) be the smallest \( \alpha_X \leq \Theta_X \) such that \( \alpha \in \sdordinal{X} \) for all limit \( \alpha < \alpha_X \) with \( \cof{\alpha}= \omega \). By Theorems~\ref{wadgeZ_eq_wadgeBaire_onedir} and~\ref{countable_cof_and_rkcomp} we get:

\begin{corollary}[\( \AD \)]\label{countablecof_ordinal_intro}
Let \( X \)	be an arbitrary zero-dimensional Polish space. Then 
either \( \alpha_X = \Theta_X = \Theta \) (if \( \kerp{X} \) is not compact), or else \( \alpha_X < \omega_1 \) and \( \alpha \in \nsdordinal{X} \) for all limit \( \alpha_X \leq \alpha < \Theta_X \).
\end{corollary}

Indeed,
by Theorems~\ref{wadgeZ_eq_wadgeBaire_onedir} and~\ref{countable_cof_and_rkcomp} the ordinal \( \alpha_X \) can be computed as follows:
\begin{itemizenew}
\item
If \( \kerp{X} \) is not compact, then \( X \) is uncountable and \( \alpha_X = \Theta_X = \Theta \).
\item
If \( \kerp{X} \) is compact and \( \rkcomp{X} \) is a successor ordinal, then \( \rkcomp{X} = \lambda + n \) for some limit or null \( \lambda < \omega_1 \) and \( n \in \omega \setminus \{ 0 \} \): set \( \alpha_X = \lambda +1 \) if \( \lambda > 0 \) and \( \alpha_X =0 \) otherwise.
\item
If \( \kerp{X} \) is compact and \( \rkcomp{X} = \lambda \) is limit, then we have to distinguish cases depending on whether \( \rkcomp{X} \in \sdordinal{X} \) or not. In the former case, we again set \( \alpha_X = \lambda+1 \). Otherwise we set \( \alpha_X = 0 \) if \( \lambda = \omega \), \( \alpha_X = \lambda'+1 \) if \( \lambda \) is of the form \( \lambda'+ \omega \) for some limit \( \lambda'< \omega_1 \), and \( \alpha_X = \lambda \) otherwise (i.e.\ if \( \lambda \) is a limit of limit ordinals).
\end{itemizenew}
In particular, \( \alpha_X \) is either \( 0 \), or the successor of a countable limit ordinal,  or a limit of countable limit ordinals, or the ordinal \( \Theta \). Moreover, when \( \alpha_X < \omega_1 \) we have \( \alpha_X \leq \rkcb{X} \) unless \( \rkcb{X} \) is limit, in which case we might have \( \alpha_X \leq \rkcb{X} \) or \( \alpha_X = \rkcb{X}+1 \).

\subsection{Length of the Wadge hierarchy for countable spaces} \label{subsec:length}

By \cref{prop:decompositionofslefdual} we know that if a zero-dimensional Polish space \( X \) is uncountable, then \( \Theta_X = \Theta \). Thus to complete the proof of part \ref{thm:maintheorem-5} of \cref{thm:maintheorem} we just need to compute the value of \( \Theta_X \) when \( X \) is countable.
Notice that since we consider only nonempty spaces \( X \), the minimal possible length for \( \wadge{X} \) is \( 2 \).

\begin{theorem}\label{length_countable_hierarchy_intro}
	Let $X$ be a 
	countable Polish space. 
	Then $\Theta_X = 2 \cdot \rkcb{X} + \varepsilon_X$, for some $\varepsilon_X \in \{-1,0,1\}$.
\end{theorem}

\begin{proof}
Let $\rkcb{X} = \lambda + n$ for some $\lambda$ limit or null and $n\in\omega$. We distinguish various cases.

If \( \lambda = 0 \) (i.e.\ \( \rkcb{X} \) is finite), then \( n > 0 \) because \( X \neq \emptyset \). Then by \cref{cor:diffclassesinuncountable} 
the largest nonselfdual Wadge classes in \( \wadge{X} \) are \( \differenceopenX{n-1}{X} \) and its dual, and by the alternating property (and the fact that the Wadge rank starts by \( 1 \)) we have \( \rank{\differenceopenX{n-1}{X}} = 2n-1 \). By \cref{countable_space_each_set_is_difference} and \cref{simple_is_NSD}, it follows that \( \Theta_X = 2n = 2 \cdot \rkcb{X} \) if \( X \) is simple, in which case we set \( \varepsilon_X = 0 \), and \( \Theta_X =  2n +1= 2 \cdot \rkcb{X}+1 \) otherwise, in which case \( \varepsilon_X = 1 \).

Assume now that \( \lambda > 0 \). By the alternating property and \cref{cor:diffclassesinuncountable} we know that \( \Theta_X \geq \lambda \) and that the ranks of difference classes of the form \( \differenceopenX{\beta}{X} \) with \( \beta < \lambda \) are cofinal in \( \lambda \). So we have two subcases. If \( n = 0 \), then \( X \) is not simple and \cref{simple_is_NSD} gives that \( \pow(X) \) is a selfdual Wadge class of rank exactly \(\lambda\): thus \( \Theta_X = \lambda+1 = 2 \cdot \rkcb{X}+1 \) and we set \( \varepsilon_X = 1 \). If instead \( n > 0 \), then by the alternating property and \cref{cor:diffclassesinuncountable} again we know that there are precisely \( n \)-many nonselfdual Wadge classes with rank \( \geq \lambda \), i.e.\ the classes \( \differenceopenX{\lambda+i}{X} \) for \( i < n \), which by the alternating property means that we have at least \( 2n-1 \) Wadge classes with rank \( \geq \lambda \). Then we have to consider two issues. First, if \( \alpha_X > \lambda \) then \( \lambda \in \sdordinal{X} \) and \( \rank{\differenceopenX{\lambda}{X}}  = \lambda+1\), i.e.\ we have one more (selfdual) class before \( \differenceopenX{\lambda}{X} \); if instead \( \alpha_X \leq \lambda \) then \( \rank{\differenceopenX{\lambda}{X}} = \lambda \) and there is no additional class to be taken into account. Second, 
by \cref{countable_space_each_set_is_difference} and \cref{simple_is_NSD} we know that if \( X \) is simple then 
\( \differenceopenX{\lambda+n-1}{X} \) is the last Wadge class in \( \wadge{X} \), while if \( X \) is not simple then
\( \pow(X) \) forms yet one more Wadge class immediately after \( \differenceopenX{\lambda+n-1}{X} \) and its dual. Summing up the discussion above we have that \( \Theta_X = 2 \cdot \rkcb{X} + \varepsilon_X \) where \( \varepsilon_X = \beta_X + \delta_X -1 \) and \( \beta_X = 0 \) if \( X \) is simple and \( \beta_X = 1 \) otherwise, while \( \delta_X = 0 \) if \( \alpha_X \leq \lambda \) and \( \delta_X = 1 \) otherwise.
\end{proof}

\subsection{Optimality} \label{section_sharpness}

The previous results are sharp, meaning that all possible shapes for a Wadge hierarchy \( \wadge{X} \) which are coherent with our description are actually realized by some zero-dimensional Polish space \( X \). We actually prove a finer result, showing that all possible configurations of the parameters \( \rkcb{X} \), \( \rkcomp{X} \), and, if \( X \) is countable,
 \( \cbtype(X) \) can be realized in suitable spaces, together with some additional features relevant to our analysis of \( \wadge{X} \). The limitations in the statement follow from the possible values of the CB-rank and the CB-type (and the fact that we consider only nonempty spaces), the definition of the compact rank, \cref{decomposition_in_simple_clopen}, and \cref{rmk:countable_cof_and_rkcomp}. Notice also that the additional conditions on \( \beta \) and \( \gamma \) are relevant only in part~\ref{prop:sharpness-1}: in particular, in part~\ref{prop:sharpness-2} the ordinal \( \gamma \) can unconditionally assume any value \( \leq \alpha \).

\begin{proposition} \label{prop:sharpness}
Let \( (\alpha,\beta) \in \omega_1 \times (\omega+1) \) and \( \gamma \leq \alpha \). Further assume that \( \beta \geq 1 \), \( \beta = \omega \) if \( \alpha \) is limit, \( \gamma = \alpha \) if \( \beta = \omega \), and \( \gamma \leq \alpha' \) if \( \alpha = \alpha'+1 \) and \( \beta < \omega \). 
\begin{enumerate-(1)}
\item \label{prop:sharpness-1}
If \( \alpha \geq 1 \) there is a countable (zero-dimensional) Polish space \(X \) with \( \cbtype(X) = (\alpha,\beta) \) and \( \rkcomp{X} = \gamma \). Moreover, when \( \gamma \) is limit we can require to have \( \gamma \in \sdordinal{X} \), or else to have \( \gamma \in \nsdordinal{X} \) (provided that \( \alpha > \gamma \), in which case \(\alpha\) needs to be a successor ordinal).
\item \label{prop:sharpness-2}
There is an uncountable zero-dimensional Polish space \( X \) with \( \kerp{X} \) compact, \( \rkcb{X} =\alpha \), and \( \rkcomp{X} = \gamma \). Moreover, when \( \gamma \) is limit we can freely decide to have \( \gamma \in \sdordinal{X} \) or \( \gamma \in \nsdordinal{X} \).
\item \label{prop:sharpness-3}
There is a(n uncountable) zero-dimensional Polish space \( X \) with \( \kerp{X} \) non-compact and \( \rkcb{X} = \alpha \). 
\end{enumerate-(1)}
\end{proposition}

\begin{proof}
\ref{prop:sharpness-1}
Further developing \cref{xmp:countable_cof_and_rkcomp}, for any \( \delta < \omega_1 \) we set
\[ 
Y_\delta = \bigoplus_{n \in \omega} K_{\delta_n+1} \qquad \text{and} \qquad Z_\delta = \ptglsq{(\omega \times K_{\delta_n+1})_{n \in \omega}},
 \] 
where \( (\delta_n)_{n \in \omega} \) is a sequence cofinal in \( \delta \) if the latter is limit, or the constant sequence with value \( \delta' \) if \( \delta = \delta'+1 \). Then
\( \cbtype(Y_\delta) = (\delta,\omega) \), \( \cbtype(Z_\delta) = (\delta+1,1) \), and \( \rkcomp{Y_\delta} = \rkcomp{Z_\delta} = \delta \). Moreover, if \( \delta \) is limit then \( \delta \in \sdordinal{Y_\delta} \) while \( \delta \in \nsdordinal{Z_\delta} \). 

The space \( X =  Y_\alpha \) already settles the case when \( \alpha \) is limit or when \(\alpha\) is a successor but \( \beta = \omega \). So assume that \( \alpha = \alpha' +1 \) and \( 1 \leq \beta < \omega \), in which case \( \gamma \leq \alpha' \). If \( \gamma = \alpha' \) then \( X = (\beta \times K_\alpha) \oplus Y_{\alpha'} \) 
has CB-type \( (\alpha,\beta ) \) and compact rank \( \gamma \), as required.
For the additional part, notice that if \( \alpha' = \gamma \) is limit, this construction gives \( \gamma \in \sdordinal{X} \); if instead we want \(\gamma \in \nsdordinal{X} \), then we use \( X = \beta \times Z_{\alpha'} \), which has still CB-type \( (\alpha,\beta) \) and compact rank \( \gamma \) but is such that \( \gamma \in \nsdordinal{X} \), as desired. The remaining cases \( \gamma < \alpha' \) are treated similarly: we set \( X = (\beta \times K_\alpha) \oplus Y_{\gamma} \) if \( \gamma \) is successor or if \( \gamma \) is limit and we want \( \gamma \in \sdordinal{X} \), while we set \( X = (\beta \times K_\alpha) \oplus Z_\gamma \) if \( \gamma \) is limit and we want \( \gamma \in \nsdordinal{X} \).

\ref{prop:sharpness-2}
Consider the space \( \hat{\cantor} = {}^\omega\{ 0,2 \} \subseteq \baire \), which is trivially homeomorphic to \( \cantor \), and recall the spaces \( Y_\delta \) and \( Z_\delta \) from part~\ref{prop:sharpness-1}. If \( \gamma < \alpha \), then we first build%
\footnote{Here we really mean the union of the two spaces, which overlap on the point \( \inftycostbaire{0} \).}
\( \hat{\cantor}_\alpha = \hat{\cantor} \cup K_{\alpha+1} \), so that \( \rkcb{\hat{\cantor}_\alpha} = \alpha \) and \( \kerp{\hat{\cantor}_\alpha} = \hat{\cantor} \) is compact, and then set \( X = \hat{\cantor}_\alpha \oplus Y_\gamma \) or \( X = \hat{\cantor}_\alpha \oplus Z_\gamma \) depending on whether we want \( \gamma \in \sdordinal{X} \) or not in case \( \gamma \) is limit. If instead \( \gamma =  \alpha \), then we set \( X = \hat{\cantor} \oplus Y_\gamma \) or \( X = \hat{\cantor} \cup Z_\gamma \), again depending on the given requirement on selfduality of \(\gamma\) when it is limit.
(Notice that \( \rkcb{\hat{\cantor} \cup Z_\gamma}  = \gamma\) because \( \inftycostbaire{0} \) is no longer isolated in \( \derrkcb{\gamma}{\hat{\cantor} \cup Z_\gamma} = \hat{\cantor} \).)

\ref{prop:sharpness-3}
Set \( X = \baire \oplus \hat{\cantor}_\alpha \), where \( \hat{\cantor}_\alpha \) is as in part~\ref{prop:sharpness-2}.
\end{proof}

\section{Proof of \texorpdfstring{\cref{thm:maintheorem2}}{Main Theorem 2}} \label{sec:mainthm2}

The Wadge hierarchy \( \wadge{X} \) over a (zero-dimensional) Polish space provides the finest possible measurement of the topological complexity of its subsets. Slightly changing our perspective, it is natural to ask whether it is possible to classify zero-dimensional Polish spaces up to the measurement scale \( \wadge{X} \) they give rise to. More formally, consider the standard Effros-Borel space
\( F(\baire) \) of all closed subsets of \( \baire \), which up to homeomorphism contains all zero-dimensional Polish spaces, and for \( X,Y \in F(\baire) \) set
\[ 
X \approx_\w Y \iff \wadge{X} \cong \wadge{Y}.
 \] 
 
 \begin{remark}
 The current formulation of the problem is within our axiomatic setup \( \ZF + \DC(\mathbb{R}) + \AD \). If one wants to dispense with any determinacy assumption, the problem could be formulated by requiring that \( \wadge{X} \) and \( \wadge{Y} \) be isomorphic when restricted to Borel sets, or even just to sets in \( \boldsymbol{\Delta}^0_2 \): all results we are going to present would be unaffected by such a change,
 thus providing determinacy-free anti-classification results.
 \end{remark}
 
The relation \( \approx_\w \) is strictly coarser than closed biembeddability:  by \cref{lem:biembimpliessamewadgehierarchy}, if \( X \) and \( Y \) are such that there are closed embeddings from each space into the  other one, then \( X \approx_\w Y \); but on the other hand \( \baire \approx_\w \omega \times \cantor \) even though \( \baire \) is not homeomorphic to a closed subset of \( \omega \times \cantor \).

Our analysis provides a quite satisfactory solution to the classification problem associated with \( \approx_\w \) by providing a natural assignment of complete invariants.
More precisely,
call a pair of ordinals \( (\alpha,\beta) \) a \emph{Wadge invariant} if it satisfies the following conditions:
\begin{itemizenew}
\item
\( \alpha,\beta \in \omega_1 \cup \{ \Theta \} \) and \( \alpha \leq \beta \);
\item
if \( \beta \neq \Theta \), then \( \beta \) is at least \( 2 \) and is a successor ordinal;
\item
if \( \alpha \neq \Theta \), then either \( \alpha = 0 \), or \( \alpha = \lambda+1 \) for some countable limit ordinal \( \lambda \), or else \( \alpha = \lambda \) for some countable ordinal \( \lambda \) which is a limit of limit ordinals.
\end{itemizenew}
Notice that each Wadge invariant can easily be coded into a single countable ordinal.
For each zero-dimensional Polish space \( X \), the pair \( (\alpha_X, \Theta_X ) \) is a Wadge invariant which completely determines the structure of \( \wadge{X} \) up to order-isomorphism  (by the alternating property, \cref{uncountablecof_ordinal_intro}, and \cref{countablecof_ordinal_intro}). This means that the map \( X \mapsto (\alpha_X, \Theta_X) \) is an assignment of complete invariants
 for the classification of zero-dimensional Polish spaces up to \( \approx_\w \). 
 
  \begin{remark}
 As briefly discussed in Footnote~ \ref{canonicaliso} on page~\pageref{canonicaliso}, in the \( \AD \)-world we must make sure that when \( (\alpha_X,\Theta_X) = (\alpha_Y,\Theta_Y) \) then we can actually find an isomorphism between \( \wadge{X} \) and \( \wadge{Y} \) without appealing to the Axiom of Choice \( \AC \). But this is granted by the fact that by \cref{relativization}\ref{relativization-3}, \cref{cor:diffclassesinuncountable}, and \cref{lem:relativization}, if \( X \subseteq \baire \) is closed and \( \rho \colon \baire \to X \) is a retraction then the map \( \iota_X \) defined by \( \nsdclass \mapsto \rho^{-1}(\nsdclass) \) is an embedding of \( \langle \wc{X}, { \subseteq} \rangle \) into \( \langle \wc{\baire}, { \subseteq } \rangle \) sending \( \nsdc{X} \) into a suitable initial segment of \( \nsdc{\baire} \).
 By our analysis, it then turns out that \( \iota_Y^{-1} \circ \iota_X \) is the required canonical isomorphism between \( \langle \wc{X}, { \subseteq } \rangle \) and \( \langle \wc{Y}, { \subseteq } \rangle \).
 \end{remark}
 
Conversely, by \cref{prop:sharpness} together with the computation of \( \alpha_X \) provided after \cref{countablecof_ordinal_intro} and the computation of the length of \( \wadge{X} \) from (the proof of) \cref{length_countable_hierarchy_intro}, we get a zero-dimensional Polish space \( X \) with \( (\alpha_X, \Theta_X) = (\alpha,\beta) \) for every Wadge invariant \( (\alpha,\beta) \). In particular, the \( \approx_\w \)-equivalence classes can be well-ordered in order type \( \omega_1 \).

In contrast to these ``positive'' classification results, we are now going to show that, except for the case when \( X \) is a singleton, it is difficult to determine whether \( X \) belongs to a given \( \approx_\w \)-equivalence class: given a Wadge invariant \( (\alpha,\beta) \neq (0,2) \), the 
set of zero-dimensional Polish spaces \( X \) with \( (\alpha_X , \Theta_X) = (\alpha,\beta) \) is never Borel. This proves \cref{thm:maintheorem2}.

\begin{theorem}[\( \AD \)] \label{thm:complexityofeqclasses}
\begin{enumerate-(1)}
\item \label{thm:complexityofeqclasses-1}
The class \( [\baire]_{\approx_\w} \) corresponding to the Wadge invariant \( (\Theta,\Theta) \) is a Borel complete \( \boldsymbol{\Sigma}^1_1 \) set.
\item \label{thm:complexityofeqclasses-2}
For any \( X \in F(\baire) \) with \( |X| > 1 \) and \( X \not\approx_\w \baire \), the set \( [X]_{\approx_\w} \) is Borel \( \boldsymbol{\Pi}^1_1 \)-hard.
\item \label{thm:complexityofeqclasses-3}
For every
countable%
\footnote{For completeness, one might ask how complex is the set \( \{ X \in F(\baire) \mid \alpha \in \sdordinal{X} \} \) for a given \( \alpha \geq \omega_1 \) of countable cofinality. But this is precisely the set \( [\baire]_{\approx_\w} \) from~\cref{thm:complexityofeqclasses}\ref{thm:complexityofeqclasses-1}.} 
limit ordinal \(\alpha\), the set \( \{ X \in F(\baire) \mid \alpha \in \sdordinal{X} \} \) is Borel \( \boldsymbol{\Sigma}^1_1 \)-hard.
\end{enumerate-(1)}
\end{theorem}

As the proof will show, the same is true if we restrict the ambient space to the analytic space \( F_{> \aleph_0}(\baire) \) of \emph{uncountable} zero-dimensional Polish spaces, i.e.\ to spaces with Wadge invariant of the form \( (\alpha,\Theta) \). (Just replace the Borel map \( F \) with \( F_\cantor \) everywhere.)

\begin{proof}
We use the following variation of a construction from~\cite{CamMarMot2018ab}.
Fix a bijection \( \langle \cdot , \cdot \rangle \colon \omega^2 \to \omega \setminus \{ 0 \} \). For \( s \in {}^{<\omega}(\omega \setminus \{ 0 \}) \) let \( \pi(s) = t \) if and only if \( \lh(t) = \lh(s) \) and for all \( i < \lh(s) \) there is \( n_i \in \omega \) such that \( s(i) = \langle t(i), n_i \rangle \). Define the Borel map \( F \colon \mathrm{Tr} \to F(\baire) \) from the Polish space \( \mathrm{Tr} \) of trees on \(\omega\) into the space of (codes for) zero-dimensional Polish spaces by letting \( F(T) \) be the body of the pruned tree
\[ 
\{ s \conc (0)^k \mid s \in {}^{< \omega}(\omega \setminus \{ 0 \} ) \wedge \pi(s) \in T \wedge k \in \omega \}.
 \] 
 If \(  x \in \baire \) is an infinite branch through \( T \), then \( F(T) \) contains a closed set homeomorphic to \( \baire \), namely, the set of all \( y \in \baire \) such that \( y \restriction n \in {}^{< \omega}(\omega\setminus \{ 0 \}) \) and \( \pi(y \restriction n) = x \restriction n \) for all \( n \in \omega \). Hence \( F(T) \) is not \(\sigma\)-compact and \( \kerp{F(T)} \) is not compact. Conversely, if \( T \) is well-founded then all elements of \( F(T) \) are of the form \( s \conc (0)^\infty \) for some \( s \in {}^{< \omega}(\omega \setminus \{ 0 \}) \) and they are isolated in \( F(T) \), i.e.\ \( F(T)\) is discrete.
 
 To prove~\ref{thm:complexityofeqclasses-1}, first recall that by \cref{cor:wadgeZ_eq_wadgeBaire_onedir} we have \( X \approx_\w \baire \) if and only if \( \kerp{X} \) is not compact. But \( \kerp{X} \) is not compact if and only if there is \( H \in F(\baire) \) such that \( H \subseteq X \), \( H \) is perfect, and \( H \notin K(\baire) \), where \( K(\baire) \) is the collection of all compact subspaces of \( \baire \). Since \( K(\baire) \) is Borel in \( F (\baire) \) and ``\( H \) is perfect'' is a Borel statement (see the proof of~\cite[Theorem 27.5]{kechris}), this is a \( \boldsymbol{\Sigma}^1_1 \) definition of \( [\baire]_{\approx_\w} \). For the hardness part, it is enough to use the map \( F \) from the previous paragraph, recalling that the set of ill-founded trees is a complete \( \boldsymbol{\Sigma}^1_1 \) set.
 
Consider now a space \( X \) as in~\ref{thm:complexityofeqclasses-2}, and let \( F_X \colon \mathrm{Tr} \to F(\baire) \) be the Borel map defined by \( F_X(T) = F(T) \oplus X \), where \( F \) is as in the first paragraph. If \( T \) is ill-founded then \( F_X(T) \approx_\w \baire \) because \( \kerp{F_X(T)} \supseteq \kerp{F(T)} \)  is not compact. If instead \( T \) is well-founded, then \( F(T) \) is a discrete clopen subset of \( F_X(T) \). Then \( \rkcb{F_X(T)} = \rkcb{X} \), and (when this makes sense) \( F_X(T) \) is simple if and only if so is \( X \) because we assumed \( |X| > 1 \). Finally, \( \rkcomp{F_X(T)} = \max \{ 1, \rkcomp{X} \} \), hence \( \alpha_{F_X(T)} = \alpha_X \). It follows that \( F_X(T) \approx_\w X \). This shows that \( F_X \) reduces the set of well-founded trees to \( [X]_{\approx_\w} \), hence the latter is Borel \( \boldsymbol{\Pi}^1_1 \)-hard.

Part~\ref{thm:complexityofeqclasses-3} follows from part~\ref{thm:complexityofeqclasses-2}; for example, one can use the map \( F_{\cantor} \). 
\end{proof}

\begin{remark} \label{rmk:complexityofeqclasses}
\cref{thm:complexityofeqclasses} admits several variations. For example, for all relevant \( \alpha \) the set \( \{ X \in F(\baire) \mid \alpha_X = \alpha \} \) is either Borel \( \boldsymbol{\Sigma}^1_1 \)-hard (if \( \alpha = \Theta \)) or Borel \( \boldsymbol{\Pi}^1_1 \)-hard (if \( \alpha \) is countable). Similarly, for all relevant \(\beta\) the set \( \{ X \in F(\baire) \mid \Theta_X = \beta \} \) is either Borel \( \boldsymbol{\Sigma}^1_1 \)-hard (if \( \beta = \Theta \)) or Borel \( \boldsymbol{\Pi}^1_1 \)-hard (if \( \beta < \omega_1 \)).
\end{remark}

\cref{thm:complexityofeqclasses} implies that
the equivalence relation \( \approx_\w \) is neither analytic nor co-analytic. Therefore, although we obtained a quite simple and satisfactory classification in terms of Wadge invariants (which are, essentially, countable ordinals), there is no reasonable classification for  \( \approx_\w \) in terms of, say, Borel reducibility.  In a sense, this hints at the fact that there are natural situations in which the latter is not an accurate tool to determine the complexity of the classification problem at hand.

The situation radically changes if we instead consider only \emph{perfect} zero-dimensional Polish spaces.

\begin{proposition}[\( \AD \)] \label{prop:complexityofeqclasses}
Let \( F_{\mathrm{p}}(\baire) \) be the standard Borel space of perfect zero-dimensional Polish spaces, and let \( X \in F_{\mathrm{p}}(\baire) \). Then either \( X \approx_\w \baire \) or \( X \approx_\w \cantor \), and those two \( \approx_\w \)-equivalence classes are Borel.
\end{proposition}

\begin{proof}
All spaces \( X \in F_{\mathrm{p}}(\baire) \) are uncountable, and thus \( \Theta_X = \Theta \) by \cref{prop:decompositionofslefdual}. Moreover, we have \( \kerp{X} = X \), thus \( X \approx_\w \baire \) if \( X \) is not compact (\cref{cor:wadgeZ_eq_wadgeBaire_onedir}), while \( X \approx_\w \cantor \) if \( X \) is compact (\cref{prop:isomorphicWadgethroughretractions}\ref{prop:isomorphicWadgethroughretractions-2}). Finally, the distinction between the two cases is Borel because \( K(\baire) \)  is Borel in \( F(\baire) \).
\end{proof}

For the record, we also notice that the
proof of \cref{thm:complexityofeqclasses}\ref{thm:complexityofeqclasses-1} actually gives a seemingly unrelated complexity result that might be of independent interest.

\begin{proposition}
The set \( \{ X \in F(\baire) \mid \kerp{X} \text{ \emph{is compact}} \} \) is a Borel complete \( \boldsymbol{\Pi}^1_1 \) set.
\end{proposition}

\section{Some questions}\label{conclusion}

Our analysis raises a number of natural questions: we just mention a few of them which follow closely our approach.

With the exception of \( [\baire]_{\approx_\w} \), \cref{thm:complexityofeqclasses} only gives lower bounds for the complexity of the other \( \approx_\w \)-equivalence classes, but they might not be sharp.

\begin{question}
Let \( X \) be a zero-dimensional Polish space with at least two points and \( \kerp{X} \) not compact. What is the exact complexity of \( [X]_{\approx_\w} \)? What for \( X = \cantor \)? What is the complexity of the equivalence relation \( \approx_\w \) as a subset of the square \( F(\baire) \times F(\baire) \)?
\end{question}

Similar questions can be raised for Polish spaces of positive dimension, or for the natural adaptation of \( \approx_\w \) to the hyperspace \( F({}^\omega \mathbb{R}) \) of all Polish spaces.

Our results are limited to zero-dimensional Polish spaces, and it is natural to wonder what happens for other kinds of Polish spaces. As recalled in the introduction, this line of research has been pursued in many papers and by various authors, although we are still lacking a complete description of the various possibilities. But from a different perspective, the restriction to zero-dimensional Polish spaces is equivalent to considering only closed (or even just \( G_\delta \)) subspaces of \( \baire \).
One can then wonder what happens for arbitrary subspaces of \( \baire \), starting with the Borel ones. This is a quite different problem, as we retain zero-dimensionality (together with separability and metrizability) but drop the completeness of the space. We leave this research direction open for future investigations.


\end{document}